\documentclass[12pt]{article}
\usepackage[latin1]{inputenc}
\usepackage{amssymb}
\usepackage{amsmath}
\usepackage{amsthm}
\usepackage{mathtools}
\usepackage{tikz}

\newcommand{\C}{\mathcal{C}}

\newcommand{\R}{\mathbb{R}}
\newcommand{\Q}{\mathbb Q}
\newcommand{\N}{\mathbb N}
\newcommand{\Z}{\mathbb Z}

\newcommand{\ball}{\mathbb{B}}
\newcommand{\ga}{\Gamma}

\newcommand{\gen}[1]{\langle #1 \rangle}

\DeclareMathOperator{\dist}{d}
\DeclareMathOperator{\diam}{diam}

\theoremstyle{definition}

\newtheorem{thm}{Theorem}[section]

\newtheorem{prop}[thm]{Proposition}

\newtheorem{lem}[thm]{Lemma}

\newtheorem{rem}[thm]{Remark}
\newtheorem{defi}[thm]{Definition}
\newtheorem{cor}[thm]{Corollary}

\newtheorem{ex}[thm]{Example}

\DeclareMathOperator{\axis}{axis}

\newcommand{\cA }{\mathcal A}
\newcommand{\cC }{\mathcal C}

\newcommand{\cH }{\mathcal H}
\newcommand{\cL }{\mathcal L}

\newcommand{\cP }{\mathcal P}

\newcommand{\cS }{\mathcal S}

\newcommand{\vs}{\vspace{0.3cm}}

\newcommand{\vsp}{\vspace{0.2cm}}

\date{}
\author{}

\begin{document}

\title{On the geometry of positive cones in finitely generated groups}
\author{J. Alonso, Y. Antol\'in, J. Brum, C. Rivas }
\maketitle

\abstract{
We study the geometry of positive cones of left-invariant total orders (left-order, for short) in finitely generated groups. 
We introduce the {\em Hucha property} and the {\em Prieto property} for left-orderable groups. 
We say that a group has the Hucha property if in any left-order the corresponding positive cone is not coarsely connected, and the Prieto property if in any left-order the corresponding positive cone is coarsely connected.
We show that all left-orderable free products have the Hucha property, and that the Hucha property is stable under certain free products with amalgamatation over Prieto subgroups. 
As an application we show that non-abelian limit groups in the sense of Z. Sela (e.g. free groups, fundamental group of hyperbolic surfaces, doubles of free groups and others) and non-abelian finitely generated subgroups of free $\mathbb{Q}$-groups in the sense of G. Baumslag have the Hucha property. 
In particular, this implies that these groups have empty BNS-invariant $\Sigma^1$ and that they do not have finitely generated  positive cones.
}

\vsp\vsp



\section{Introduction}

Let $G$ be a group.
A subset $P$ of $G$ is {\it a positive cone} of $G$ if it is a subsemigroup which, together with $P^{-1}$ and the identity $\{1\}$, forms a partition of  $G$. 
Elements from $P$ and $P^{-1}$ are called positive and negative respectively, and groups admitting positive cones are called {\em left-orderable} since from every positive cone $P$ a total and left-multiplication-invariant order $\prec$ (a {\em left-order}, for short) can be defined on $G$ by setting $g\prec h$ whenever  $g^{-1}h \in P$. 

Suppose that $G$ is finitely generated, and it is  endowed with a word metric. In this paper, we study the geometry of the positive cones $P$ of $G$, focusing on whether $P$ is (coarsely) connected or not\footnote{Recall that, given a  metric space $(X,\dist)$, a subset $Y\subseteq X$ is coarsely connected if there is $N$ such that $\{x\in X\mid \dist(x,Y)\leq N\}$, the $N$-neighborhood of $Y$, is connected}. As we will see below, this geometric information give us algebraic and formal-language complexity  information about the positive cone $P$. 

Our initial observation is that, although locally the geometry of a positive cone $P$ and the ambient group $G$ might be quite different, when $G$ acts on a (Gromov) hyperbolic space, both the action of $G$ and  of $P$ on the boundary  look the same.  
To be more concrete, let us fix some notation. Suppose that $G$ is acting by isometries on an hyperbolic space $\Gamma$. 
For $H$ a subset of $G$ and $x_0\in \Gamma$, the set $Orb_H(x_0)=\{h x_0\mid h\in H\}$ is called the orbit of $x_0$ under $H$ and we denote by $\Lambda(H)$ the accumulation points of $Orb_H(x_0)$ in $\partial \Gamma$ (it is easy to see that this notion is independent of $x_0$). 
The action is called {\it non-elementary} if $|\Lambda(G)|\geq 3$. 
We will say that the action is {\it of general type} if it is non-elementary and not quasi-parabolic i.e. $G$ does not fix a point of $\Lambda(G)$. (The literature about isometries of Gromov hyperbolic spaces is vast. We recomend \cite{GH} for an introduction.)
\begin{prop}\label{prop: main boundary}
Let $\Gamma$ be a Gromov hyperbolic metric space and suppose that $G\curvearrowright \Gamma$ is an action by isometries of general type.
Let $P$ be a positive cone of $G$. 
Then, the $P$-orbits in $\Gamma$ accumulate on every point of $\Lambda(G)$.
\end{prop}


By a theorem of Osin \cite[Theorem 1]{Osin} every non-elementary acylindrical action on a hyperbolic space by isometries is of general type. In particular the theorem applies to hyperbolic groups acting on their Cayley graph, relatively hyperbolic groups acting on their relative Cayley graph \cite[Proposition 5.2]{Osin}, or mapping class groups of surfaces of sufficiently high complexity acting on their curve complex, to name some classical examples.   It is worth pointing out that Koberda and Kielak had  already observed interactions between  geometry and orderability. On the one hand, Koberda \cite{koberda}  note that the intersection of positive cones containing a given element of an hyperbolic group $G$ yields a unique point in its Gromov boundary, and that the collection of all such points is
a dense subset of $\partial G$. On the other hand, Kielak \cite{kielak} exhibits an argument that recovers Proposition \ref{prop: main boundary} in the case of a group with infinitely many ends acting on its set of ends.

Let us see,  as an example,  how to use Proposition \ref{prop: main boundary} to prove that non-abelian free groups  have no coarsely connected positive cones.  
Suppose that $G=F_n$ is a free group of rank $n\geq 2$ and suppose that $\Gamma$ is the Cayley graph of $G$ with respect to a basis. 
In this case, $\Gamma$ is a tree whose vertices are in correspondence with the elements of $G$ and $G$ acts on $\Gamma$ by left-multiplications. 
Fix a positive cone $P$ with corresponding left-order $\prec$.
Then for every $R>0$, there is $g_R\in G$ which is larger (with respect to $\prec$) than every element from  $\ball(1_G,R)$, the ball of radius $R$ in $G$. By left-invariance, every element of $g_R^{-1}\ball(1,R)=\ball(g_R^{-1},R)$ is smaller than $1_G$ and  we conclude that $\ball(g_R^{-1},R)$  is contained in $P^{-1}$. 
Finally, by Propostion \ref{prop: main boundary}, every component of $\Gamma \setminus\ball(g_R^{-1},R)$ contains at least one positive element, and so we can find two elements of $P$ such that any path connecting them must go through $\ball(g_R^{-1},R)$.  Since $R$ is arbitrary, we conclude that $P$ is not a coarsely connected subset of $\Gamma$.

It is easy to see that (up to adjusting constant) the previous argument is independent of  generating sets. So we have proved the following.
\begin{cor}\label{cor: free groups}
Positive cones of non-abelian free groups  are not coarsely connected.
\end{cor}

 Corollary \ref{cor: free groups} easily follows from Kielak's work \cite{kielak} and also from the analisys of the BNS invariant of the free group (see definition below). Yet, the  previous proof exemplifies  how we will exploit the interaction of the geometry at infinity of $P$, with the local geometry of $G$.  We use a similar but more involved argument in Theorem \ref{thm: surface group}, where we show that fundamental groups of hyperbolic surfaces do not have coarsely connected positive cones. 
In this case, $G$ acts properly and co-compactly on its Cayley graph, a space which is quasi-isometric to the hyperbolic plane $\mathbb{H}^{2}$. 
Finding sets of negative elements separating the boundary of $\mathbb{H}^2$ (we call such sets  {\it negative swamps}) becomes more delicate, and we heavily relies on the ``planarity" of $\mathbb{H}^2$.

The situation for hyperbolic groups in general is much more complex, as there exists examples of groups with connected positive cones. 
To see this, recall that given a finitely generated group $G$, a non-trivial homomorphism $\phi\colon G\to \R$ belongs to $\Sigma^1(G)$, the Bieri-Neumann-Strebel invariant (BNS invariant for short), if and only if  $\phi^{-1}((0,\infty))$ is connected. Moreover, the kernel of $\phi$ is finitely generated if and only if both $\phi$ and $-\phi $ belong to $\Sigma^1(G)$ (see \cite{BNS}).  
 In this light, it is easy to come up with examples of hyperbolic groups with non-trivial BNS invariant, for instance by considering (finitely generated free)-by-$\Z$ groups that are hyperbolic (see \cite{Brink} for a characterization) or, up to passing to finite index, any fundamental group of a closed hyperbolic 3-manifold (see \cite{Agol}). 
 Clearly, any (left-orderable)-by-$\Z$ group is left-orderable (since left-orderability is stable under extensions, see for instance \cite{GOD}), and we can show that all the previous examples support connected positive cones. Indeed, in Section \ref{sec: basics} we provide an easy argument showing the following.

\begin{prop}\label{prop:BNS}
Every finitely generated left-orderable group $G$ with non-empty $\Sigma^1(G)$ admits connected positive cones.
In particular, there are left-orderable hyperbolic groups with connected positive cones.
\end{prop}

This may seem counter-intuitive in view of Proposition \ref{prop: main boundary}, yet, we will show that when positive cones of hyperbolic groups are coarsely connected, they have to be very distorted.
More precisely, in  Theorem \ref{thm combing},  we show that if a positive cone in a non-elementary hyperbolic group is coarsely connected, then for every $\lambda \geq 1$, $c,$ and $r\geq 0$ there are pairs of positive elements that {\em cannot} be joined by a $(\lambda, c)$-quasi-geodesic supported on the $r$-neighourhood of  the positive cone.

\begin{rem} \label{rem: tilde T}It is tempting to believe that the converse of Proposition \ref{prop:BNS} holds in general, namely that having a coarsely connected positive cone implies that the BNS invariant is not trivial. This is, however, not true in general. For instance there are perfect groups (i.e. groups with trivial abelianization) which supports only coarsely connected positive cones. A concrete example is $\widetilde{T}=\langle a,b,c\mid a^2=b^3=c^7=abc\rangle$, which is easily seen to be a perfect and left-orderable group such that all its positive cones are connected (see Corollary \ref{cor: Braids are Prieto}).
\end{rem}

Coming back to applications of Proposition \ref{prop: main boundary}, we can show that many HNN extensions and amalgamated free products do not admit coarsely connected positive cones by exploiting their action on their associated Bass-Serre trees. 
Curiously, to show that certain groups acting on trees do not admit coarsely connected positive cones, we will need the opposite property for edge stabilizers, namely that for (certain) edge stabilizers all their positive cones are coarsely connected.

Before stating our main result, we need to set some terminology.
\begin{defi}
Let $G$ be a finitely generated left-orderable group endowed with a word metric with respect to a finite generating set and let $H$ be a subgroup of $G$. 
\begin{enumerate}
\item We say that $G$ is {\it Prieto} if all its positive cones are coarsely connected.
\item We say that $G$ is {\it Hucha with respect to $H$} if for all $r>0$ and all positive cones $P$ of $G$, the $r$-neighbourhood of $P$ is not connected  and there are cosets $g_1H$ and $g_2H$ that cannot be connected by a path inside the $r$-neighbourhood of $P$.

\end{enumerate}
We will say that a group is Hucha if it is Hucha relative to the trivial subgroup.\footnote{About the terminology. Prieto is an Spanish word which could be translate as tight. Whereas hucha refers to an old leather gadget who has an elongated slot which opens under a small pressure in order to insert and keep the coins.    }
\end{defi}

Prototypical examples of Prieto groups are finitely-generated torsion-free abelian groups.  It is well known that if one view $\Z^n$ as the standard lattice in $\mathbb{R}^n$, positive cones essentially correspond to half-spaces defined by hyperplanes in $\R^n$ going through the origin (see Section \ref{sec: Prieto}).  
Less obvious examples are braid groups and the group $\widetilde{T}$ from Remark \ref{rem: tilde T}, among others (see Corollary \ref{cor: Braids are Prieto}).  
On the other hand,  free groups are {\it Hucha groups} with respect to all its infinite index finitely generated subgroups  (see Proposition \ref{prop:F2hucha}).

As mentioned before, there is  a beautiful interplay between the opposite properties Prieto and Hucha, that allows to construct new Hucha groups.  For instance we can show the following.

\begin{thm}\label{thm: main combination}
Suppose that $A$ has the Hucha property with respect  to a subgroup $C$, and suppose that $C$ is Prieto. 
Let $G$ be either an HNN extension $A*_C t$ of an amalgamated free product of $A*_C B$, where $B$ is any finitely generated  group.
If $G$ is left-orderable, then $G$ is Hucha.
\end{thm}

Theorem \ref{thm: main combination} is really a corollary of our main technical result, Theorem \ref{thm:induction}.
There, we conclude that HNN extensions and amalgameted free product are Hucha with respect to a  family of non-trivial subgroups. 
Starting from free groups, we use  Theorem \ref{thm:induction} to inductively build an infinite family of Hucha groups that we denote by $\mathfrak{H}$.

\begin{defi}\label{def: familia H}
Let $\mathfrak{H}$ be the smallest family of groups containing  non-abelian finitely generated free groups,  that is closed under taking free products and taking non-abelian finitely generated subgroups, and such that for every $G\in \mathfrak{H}$ and any cyclic centralizer subgroup $C$ of $G$, the group $G*_C A$, where $A$ is finitely generated torsion-free abelian, lies in $\mathfrak{H}$.
\end{defi}
It easily follows from a theorem of Howie \cite{Howie} (see Proposition \ref{prop:howie}) that every group in $\mathfrak{H}$ is locally indicable (i.e. every non-trivial finitely generated subgroups maps onto $\Z$), and hence left-orderable \cite{burnshale}. 
Further, there are two important families of groups closely related to our family $\mathfrak H$. 
One is the family of   {\em limit groups} introduced by Sela (see \cite{ChampGuir} for a nice survey). 
Indeed, it follows from the hierarchical characterization given by  Kharlampovich and Mysniakov \cite{KM98} that  {\em non-abelian limit groups} belong to $\mathfrak H$.  The second, is the family of {\em free $\Q$-groups} introduced by G. Baumslag \cite{GBaumslag}.
 It turns out that free $\Q$-groups also admit a hierarchical construction starting from free groups, from which we can   deduce that {\em finitely generated subgroups of free $\Q$-groups} are also contained in $\mathfrak H$. See Section \ref{sec: limit} for details.


The main theorem of this paper is the following.
\begin{thm}\label{thm: main}
Any group in $\mathfrak{H}$ is Hucha. In particular, non-abelian limit groups and finitely generated subgroups of free $\Q$-groups are Hucha.
\end{thm}

In light of Proposition \ref{prop:BNS}, we obtain the following immediate consequence of Theorem \ref{thm: main}.

\begin{cor} The BNS-invariant $\Sigma^1$ of every group in $\mathfrak{H}$ is trivial. In particular, the BNS-invariants $\Sigma^1$ of non-abelian limit groups and finitely generated subgroups of free $\Q$-groups are trivial.
\end{cor}

 We point out that for the case of non-abelian limit groups this was already known by Kochloukova, who proved it with homological tools \cite[Lemma E]{Kouch}.
 Also, the emptyness of $\Sigma^1$ for non-abelian limit groups can also be derived from a results  of $\ell^2$-Betti numbers. Hillman \cite[Theorem 6]{Hillman} proved that a necessary condition for a finitely presented group to have a non-empty $\Sigma^1$ is to have vanishing  first $\ell^2$-Betti number equal to zero, however Pichot \cite[Theorem 1]{Pichot} proved that these is not the case for non-abelian limit groups.

Another immediate consequence of the lack of coarse connectivity of a positive cone is that it can not be finitely generated as a semi-group (indeed see Remark \ref{rem: fg pos cone are coarsely connected}). 
The search for groups supporting or not supporting finitely generated positive cones is an active line of research (see \cite{GOD} and reference therein for some partial account). 
For instance, using dynamical techniques it was shown that free-products of groups \cite{rivas free} and fundamental groups of closed hyperbolic surfaces \cite{ABR} do not support finitely generated positive cones (in fact, this method proves something stronger:  it rules out the existence of {\em isolated } left-orders). This dynamical approach, roughly, consists of perturbing the dynamics of the generator of a group acting on the line while preserving its relations,
a task that became impractical when dealing with complicated presentations (which is the case, for instance, of limit groups that are high in the hierarchy). 
In contrast, our method immediately yields the following.


\begin{cor} 
No positive cone of a group in $\mathfrak{H}$ is finitely generated as a semi-group. In particular, no positive cone in a non-abelian limit group or in a finitely generated subgroup of free $\Q$-group is finitely generated as a semi-group.
\end{cor}

In fact, something stronger holds for Hucha groups: they do not admit positive cones  that can be described by a regular language over the generators (regular positive cones for short). See Corollary \ref{cor: hucha not regular}. Recall that, roughly, a subset  of a finitely generated group is regular if it can be described by the set of paths in a finite labelled graph (see Section \ref{sec: regular} for a precise definition).  
Certainly, finitely generated positive cones are regular, but being regular is a much more stable property. 
For instance, regularity passes to finite index subgroups \cite{Su} while finite generation does not\footnote{An easy example is given by $\Z^2$, which has no finitely generated positive cones but it is a finite index subgroup of $K=\langle a,b\mid aba^{-1}=b^{-1}\rangle$,  which contains a positive cone generated as a semi-group by $a$ and $b$.}. 
We recommend \cite{ARS} for an introduction. 
It is  easy to show that finitely presented groups with a regular positive cone have solvable word problem (see \cite{ARS}).
A more refined criterium for not admitting regular positive cones was obtained by Hermiller and \u Suni\'c in \cite{HS} where they show that no positive cone in a free product of groups is regular. Their criterion is stated in a language-theoretical fashion and finding a geometrical interpretation of it was one of our initial motivations to pursue the present work.


The paper is organized as follows: we start in Section \ref{sec: basics} proving some basic features of the geometry of positive cones. 
Among other things, we show that a positive cone naturally defines a special geodesic on the group with the property that {\em it  goes deep} into the negative cone (see Proposition \ref{prop tool}). 
In Section \ref{sec: hyperbolic}  we recall the basics of hyperbolic geometry to show Proposition \ref{prop: main boundary}. 
In that section we also show that the Hucha property holds for surface groups (Theorem \ref{thm: surface group}) and that there is no quasi-geodesic combings for positive cones of hyperbolic groups (Theorem \ref{thm combing}). 
We remark that results of Section \ref{sec: hyperbolic} are not needed elsewhere in the paper so readers interested on Theorem \ref{thm: main} and its applications might  want to skip the section.
In Section \ref{sec: prieto and hucha}, we introduce the Prieto and Hucha properties, and show basic stability results such as the independence of the generating set that these properties passes to finite index overgroups. 
In Section \ref{sec: trees}, we will study left-orderable groups acting on trees and see how the geometry of the tree dominates the geometry of the group to show our combination theorems (Theorem \ref{thm:induction} and Theorem \ref{thm: main combination}). 
In Section \ref{sec: limit}, we will review the needed facts of limit groups and $\mathbb{Q}$-groups and show Theorem \ref{thm: main}. 
Finally, in Section \ref{sec: regular}, we recall the definition of regular languages, regular positive cones, and observe that Hucha groups do not admit regular positive cones.

\vs


\noindent{\textbf{{Acknowledgments}}} 
We are grateful to R. Potrie, S. Hermiller, Z. \u Suni\'c  and H. L. Su for their interest in this work, and  for many fruitful conversations regarding earlier drafts of this paper.
We thank G. Gardam for pointing out some useful references.
 Y.A. acknowledges partial support from the Spanish Government through grants number MTM2014-54896 and MTM2017-82690-P, and through the ''Severo Ochoa Programme for Centres of Excellence in R\&{}D'' (SEV-2015-0554),
and through European Union's Horizon 2020 research and innovation programme under the Marie Sk\l{}odowska-Curie grant agreement No 777822.
C.R. acknowledges partial
support from FONDECYT 1181548.
J.B. acknowledges support from FONDECYT Postdoc 3190719.

\section{Notation and basics facts}
\label{sec: basics}
Let $\Gamma=(V\Gamma, E\Gamma)$ be a graph. 
A {\it  path or  $1$-sequence} in $\Gamma$, is a function $\alpha\colon \{0,1,\dots, k\}\to V\Gamma$ where
$\alpha(i)=\alpha(i+1)$ or there is an edge connecting $\alpha(i)$ and $\alpha(i+1)$.
We also use the notation $\{\alpha(i)\}_{i=0}^k$ to denote $\alpha$.
We now can define the combinatorial metric on $V\Gamma$, setting $\dist_\Gamma(u,v)$ to be the minimum $k$ such that there is a path $\{v_i\}_{i=0}^k$ with $v_0=u$ and $v_k=v$.
If no such path exists we set $\dist_\Gamma(u,v)=\infty$.

By $\ball_\Gamma(v,r)$ we denote the set $\{u\in V\Gamma \mid \dist_\Gamma(u,v )\leq r\}$, the closed ball of radius $r$.

Let $r\geq 1$. 
An {\it $r$-path or $r$-sequence} $\{v_n\}_{n=1}^k$ is a sequence in $\Gamma$ such that $\dist_\Gamma(v_i,v_{i+1})\leq r$. 
Let $\lambda \geq 1, c\geq 0$. 
An $r$-path is a {\it  $(\lambda,c)$-quasi-geodesic} if  
$$\dfrac{|i-j|}{\lambda}-c \leq \dist_\Gamma(v_i, v_j) \quad \forall\, 1\leq i,j\leq k.$$
A {\it geodesic} is a $1$-path that is a $(1,0)$-quasi-geodesic.

For infinite $r$-paths we will use the term $r$-rays, $r$-quasi-geodesics or geodesic rays.

A subset $S$ of $V\Gamma$ is {\it coarsely connected} if there is some $r\geq 1$ such that for all $u,v$ in $S$ there is an $r$-sequence supported at $S$ starting at $u$ and ending at $v$. 
Equivalently, $S$ is coarsely connected if there is $r\geq 1$ such $\{v\in V\Gamma | \dist(v,S)\leq r\}$ spans a connected subgraph. 

Let $G$ be a group and $X$ a generating set of $G$. 
Generating sets, unless otherwise stated, will assumed to be symmetric, that is $X=X^{-1}$.
We denote by $\Gamma(G,X)$ the {\it Cayley graph} of $G$ with respect to $X$. 
By identifying the group $G$ with the vertices of $\Gamma(G,X)$, $G$ is endowed with a metric $\mathrm{d}_X=\dist_{\Gamma(G,X)}$. 
We also write $|g|_X$ for $\mathrm{d}_X(1_G, g)$. 
By $\ball_X(g,r)$ we denote the ball in $\Gamma(G,X)$ with center $g$ and radius $r$.
We will drop the subscripts $_X$ and $_\Gamma$ when the meaning is clear from the context. 

\subsection{General facts about the geometry of positive cones}
Recall that  $P\subseteq G$ is a {\it positive cone} of $G$, if $P$ is a sub-semigroup and $G=P\sqcup P^{-1}\sqcup \{1\}$ where $P^{-1}=\{g^{-1}:g\in P\}$. 
Moreover, $a\prec b \Leftrightarrow a^{-1}b\in P$ defines a $G$-left-invariant total order on $G$. 
Conversely, if $\prec$ is a $G$-left-invariant total order on $G$, then $P=\{g\in G \mid 1_G\prec g\}$ is a positive cone.
Groups with positive cones must be torsion-free. 

If $\prec$ is a left-order on $G$, and $X$ and $Y$ are subsets of $G$ we will write $X\prec Y$ if for all $x\in X$ and all $y\in Y$, $x\prec y$ holds. 
If $X=\{x\}$ a singleton, we will write simply $x\prec Y$ to denote $\{x\}\prec Y$.

Given $S\subseteq G$, we denote by $\langle S \rangle^+$ the sub-semigroup generated by $S$. 
A positive cone is {\it finitely generated} if there is a finite set $S$ of $G$ such that $P=\langle S \rangle^+$. 
\begin{rem}\label{rem: fg pos cone are coarsely connected}
 If $P$ is a finitely generated positive cone of a group $G$, then it is coarsely connected subset of the Cayley graph of $G$.
Indeed, let $P=\langle S \rangle^+$ with $S$ finite and $G=\langle X \rangle$. 
Let $r=\max_{s\in S} |s|_X$, then there are $r$-paths in $P$ from $1_G$ to $g\in P$ for all $g\in P$. 
\end{rem}

We show now that  positive cones define preferred ways to go to infinity and that positive cones contain arbitrarily large balls. This will be essential for showing that fundamental groups of hyperbolic surfaces do not have coarsely connected positive cones.

\begin{prop}\label{prop tool}
Let $G$ be generated by a finite symmetric set $X$. 
Let $P$ be a positive cone of $G$ with associated order $\prec.$
Let $g_n=\max_{\prec }\ball_X(1_G,n)$. The following holds:
\begin{enumerate}
\item The map $\{g_n^{-1}\}_{n=0}^\infty$ is a geodesic ray in $\Gamma(G,X)$. In particular, $|g_n|_X=n$.
\item $\ball_X(g_n^{-1},n-1) \subseteq P^{-1}$. 
\end{enumerate}
In particular $P^{-1}$ contains  $\bigcup_n \ball_X(g_n^{-1},n-1)$.
\end{prop}
\begin{proof} 
To show 1, for $0< i\leq n$ let $g_i=\max_{\prec} \ball(1_G,i)$ and assume, by induction on $n$, that $|g_i|_X=i$ and $g_i=x_i g_{i-1}$ with $x_i\in X$.
The base of induction, $n=1$, $g_1=\max_{\prec } X$ and $g_{0}=\{1_G\}$, clearly holds.
Let $g_{n+1}=\max_{\prec} \ball(1_G,n+1)$. 
By definition $|g_{n+1}|_X\leq n+1$. 
So we need to show that $g_{n+1}\notin \ball(1_G,n)$ and that there is some $x\in X$ such that $g_{n+1}=xg_n$.

By left-invariance, $xh\preceq xg_{n}$ for all $x\in X$ and all $h\in \ball(1_G, n)$.
Thus 
$$\ball(1_G, n+1)= \bigcup_{x\in X}x\ball(1_G, n)\preceq\max_{\prec}\{xg_n \mid x\in X\}.$$
Thus $g_{n+1}\in \{xg_n\mid x\in X\}$.
It remains to show that $g_{n+1}\notin \ball(1_G,n)$.
For that, by totality of the order, for each $x\in X$, either $xg_n\prec g_n$ or $g_n\prec xg_n$,
which, by left-invariance it is equivalent to $g_n\prec x g_n$ or $g_n \prec x^{-1} g_n$.
Thus $\ball(1_G, n)\preceq g_n \prec \max_{\prec}\{xg_n \mid x\in X\}$.
This completes the proof of the induction.

It follows that the map $n\mapsto g_n^{-1}$ is a geodesic ray, that is $\dist(g_n^{-1},g_m^{-1})=|m-n|$.   

To show 2, we let  $b\in \ball(1_G, n-1)$. 
By definition we have that $g_n\succ b$ and so by left-invariance we get that $1_G\succ g_n^{-1}b$.  
\end{proof}

\begin{cor}\label{cor: tool}
Let $X=X^{-1}$ be a finite  generating set of $G$ and $P$ a positive cone of $G$.
Every connected component of $P$ in $\ga(G,X)$ is infinite and contains a geodesic ray.
\end{cor}
\begin{proof}
Applying Proposition \ref{prop tool} to $P^{-1}$, we get a geodesic ray $\{g_n\}_{i=0}^{\infty}$ starting at $1_G$ and contained (except from $g_0$) in $P$.
Let $g\in P$.
By left-invariance of the action of $G$ on $\ga(G,X)$, $\{gg_n\}_{i=0}^\infty$ is a geodesic ray.
Since $P$ is a sub-semigroup $\{gg_n\}_{i=0}^\infty \subseteq P$.
\end{proof}

We finish this section by showing that left-orderable groups with non-trivial BNS invariant enjoys connected positive cones.

\begin{proof}[Proof of Proposition \ref{prop:BNS}]  
Let  $G$  be a left-orderable finitely generated  by $X=X^{-1}$, and assume that  $\Sigma^1(G)$  is non-empty. 
Then, there is a non-trivial homomorphism $\phi \colon G \to \R$ such that $\phi^{-1}((0,\infty))$ is connected. 
Let $x_0\in X$ such that $\phi(x_0)>0$. 
It is known that  $\{g\in G\mid \phi(x_0)<\phi(g)\}$ is also connected \cite{BNS}.

Now let $K=ker(\phi)$ and $\preceq_K$ be any left-ordering on $K$. 
Using $\phi$, we can lexicographically extend $\preceq_K$ to produce a left-ordering of $G$. 
Precisely we set $1\prec_G g$ if and only if $\phi(g)>0$ or $\phi(g)=0$ and $1 \prec_K g$ (checking that this is a left-order is left to the reader). 
Now, if $g_1$ and $g_2$ are positive in $\preceq_G$ then $g_1x_0$ and $g_2x_0$ both belong to $\{g\in G\mid \phi(x_0)<\phi(g)\}$, which we already pointed out that is connected. 
In particular, there is a path made of positive elements connecting $g_1$ and $g_2$. 
\end{proof}

\section{Positive cones  acting on hyperbolic spaces}
\label{sec: hyperbolic}

In this section we prove Proposition \ref{prop: main boundary}. As stated, the proof works in general for not necesarily proper (i.e. locally compact) metric spaces. Yet, in this section it will only be applied to hyperbolic groups (i.e. proper spaces), thus the reader only interested in these applications can assume that the spaces involved are proper.
We follow \cite{Hamann}.

Let $\Gamma$ be a graph.
The Gromov product of $a$ and $b$ in $V\Gamma$ (with respect to $v$) is defined by 
$$(a|b)_v=\dist(a,v)+\dist(b,v)-\dist(a,b).$$
Let $\delta \geq 0$. The graph $\Gamma$ is $\delta$-hyperbolic if for a given base-point $v\in V\Gamma$ we have
$$ (a|b)_v\geq \min \{(a|c)_v, (b|c)_v\}-\delta$$
for all $a,b,c\in V\Gamma$.
With this definition of hyperbolicity, every geodesic triangle in $\Gamma$ is $(4\delta)$-thin.
Moreover, if $\alpha$ is a geodesic from $v$ to $a$, and $\beta$ is a geodesic from $v$ to $b$ then 
\begin{equation}\label{eq: gromov product measures cancellation}
\dist(\alpha(t),\beta(t))\leq 2\delta \qquad \forall t\in [0, (a|b)_v].
\end{equation}
Finally, if $\Gamma$ is $\delta$-hyperbolic with respect to some base-point $v$, then it is $\delta(u)$-hyperbolic for any other base-point $u$.

Let $\epsilon > 0$ with $\epsilon' = \exp(\epsilon \cdot \delta) -1 < \sqrt{2} - 1$.
For $a,b\in V\Gamma$, define $\rho_\epsilon (a,b)= \exp (\epsilon \cdot (a|b)_v)$ if $a\neq b$, and $\rho_\epsilon(a,a)=0$. The function $\rho_\epsilon$ is a pseudo-metric. However we can construct a metric on $V\Gamma$ by setting
$$\dist_\epsilon(a,b)= \inf_{n}\{ \sum_{i=1}^n \rho_\epsilon (z_i, z_{i+1}) \mid z_1=a, z_{n+1}=b, z_i\in V\Gamma \} $$
and this metric satisfies that $(1-2\epsilon')\rho_\epsilon (a,b) \leq \dist_\epsilon (a,b) \leq \rho_\epsilon(a,b)$.

Let $\hat{\Gamma}$ be the metric completion of $V\Gamma$ with respect to $\dist_\epsilon$. Then, one can define the Gromov boundary $\partial \Gamma$ of $\Gamma$ as $\hat \Gamma \setminus V\Gamma$. 
The topology of the Gromov boundary is independent of the base point $v$ and the parameter $\epsilon$.

Let $G$ be a group acting by isometries on $\Gamma$. 
This action naturally extends to a continuous action on $\hat{\Gamma}$ and hence on $\partial \Gamma$.
Given a subset $H$ of $G$ and a vertex $v\in V\Gamma$, we denote by $\Lambda(H)$  the intersection of the closure of the orbit $Hv$ in $\hat{\Gamma}$ (with respect to the metric $\dist_\epsilon$) with $\partial \Gamma$.
The set   $\Lambda(H)$ is independent of $v$.

The following theorem classifies the actions of $G$ on $\Gamma$ in terms of the dynamics of $G$ on $\Lambda(G)$.

\begin{thm}\cite[Theorem 2.7]{Hamann}\label{thm: classification actions}
Suppose that $G$ acts by isometries on an hyperbolic geodesic space $\Gamma$.
Then one of the following holds.
\begin{enumerate}
\item[(i)] {\it (elliptic action)} $\Lambda(G)$ is empty.
\item[(ii)] {\it (parabolic action)}  $G$ fixes a unique point of $\Lambda(G)$.
\item[(iii)] {\it (dihedral action)} $\Lambda(G)$ has two points.
\item[(iv)] There are two elements $g,h\in G$ such that $|\Lambda( \langle g \rangle )\cup \Lambda(\langle h\rangle) |=4$.
\end{enumerate}

\end{thm}

The action is called {\it non-elementary} if $|\Lambda(G)|> 2$ (in fact this is equivalent to  $|\Lambda(G)|=\infty$).
We will say that a non-elementary action is {\it of general type} if it lies on case (iv) of the previous theorem.

By a theorem of Osin \cite[Theorem 1]{Osin} every non-elementary acylindrical action on a hyperbolic space by isometries is of general type.
Recall that the action of $G$ on $\Gamma$  is {\it acylindrical} if for every $r\geq 0$, there are $N$ and $R$ such that for any pair $a,b\in V\Gamma$ with $\dist(a,b)\geq R$, the set $\{g\in G \mid \dist(a,ga)\leq r \text{ and  } \dist(b, gb)\leq r\}$ has at most $N$ elements.  
This implies, although it is much easier to prove, that 

\begin{rem}\label{rem: hyperbolic groups are of general type}
non-elementary proper co-compact actions are of general type. See also \cite[Chapitre 8, Th\'eor\`eme 30]{GH} for a direct proof. 
\end{rem}

An element $g\in G$ is called {\it loxodromic} if $|\Lambda(\langle g\rangle)|=2$.
Equivalently, an element $g\in G$ is loxodromic if $n\mapsto g^n v$, $n\in \Z$, defines a quasi-geodesic. The {\it loxodromic limit set} of $G$, denote $\cL(G)$, is the set $\Lambda( \{g \in G \mid g\;\, \text{loxodromic} \})$. 
The loxodromic limit set is  {\it bilaterally dense in $\Lambda(G)$} if for every pair of disjoint non-empty open subsets $A$, $B$ of $\Lambda(G)$, there is  a loxodromic  element $g\in G$ such that $\Lambda(\langle g \rangle )$ has non-trivial intersection with both $A$ and $B$, in other words, the quasi-geodesic $g^nv$ goes from  $A$ to $B$.

Proposition \ref{prop: main boundary} follows from the following.
\begin{thm}\cite[Theorem 2.9.]{Hamann}
Suppose that $G$ acts non-elementary by isometries on an hyperbolic graph $\Gamma$. Then
\begin{enumerate}
\item $\Lambda(G)$ is closed without isolated points (i.e is a perfect set).
\item the hyperbolic limit set is bilateraly dense in $\Lambda(G)$ if and only if the action is of general type.
\end{enumerate} 
\end{thm}

\begin{proof}[Proof of Proposition \ref{prop: main boundary}]
Let $\eta \in \Lambda(G)$ and $A$ an open neighbourhood of $\eta$ in $\Lambda(G)$. 
Since the action is non-elementary, $\Lambda(G)$ is perfect and thus $\eta$ is not isolated and we can assume that there is $\nu\neq \eta$ in $A$.
Since the topology on $\partial \Gamma$ (and $\Lambda(G)$) is Hausdorff, there are disjoint open neighborhoods $U$ and $V$ of $\eta$ and $\nu$ respectively lying in $A$.
Since the action of $G$ is of general type, the hyperbolic limit set is biliaterally dense in $\Lambda(G)$ and thus, there is a loxodromic element $g$ such that $\Lambda(\langle g \rangle)$ meets $U$ and $V$. 
Now, if $P$ is positive cone of $G$, we have that either $g$ or $g^{-1}$ lies in $P$.
And thus either $\Lambda(\langle g \rangle^{+})$ or $\Lambda(\langle g^{-1} \rangle^{+})$ lies in $\Lambda(P)$ and thus $\Lambda(P)\cap A \neq \emptyset$.
\end{proof}

\subsection{Positive cones in surface groups}
\label{subsec surface}

As promised, we show how the ideas of the proof of Corollary \ref{cor: free groups}  can be adapted to the case of surface groups. The proof uses strongly the topology of $\mathbb{H}^2$. Specifically, that $\mathbb{H}^2$ can be cut into two infinite connected components using a curve.

\begin{thm} \label{thm: surface group}
 Suppose that $G$ is the fundamental group of a closed hyperbolic surface. Then, no positive cone in $G$ is coarsely connected.
\end{thm}

\begin{proof} 
Assume that $X$ is the standard generating set of $G$, so that the Cayley graph, not only is quasi-isometric to $\mathbb{H}^2$, but also it  is embedded as a 1-dimensional topological subspace.
In particular, any path in $G$ connecting two different points of $\partial G$ separates the Cayley graph into two infinite connected components. 
We fix a positive cone $P$ and the corresponding left-order $\prec.$  
We also fix $N_0>1$.  
By Proposition \ref{prop: main boundary}, $P$ accumulates on every point of $\partial G$, so to show that $P$ is not $N_0$-connected, it is enough to build a bi-infinite path $\gamma:\Z\to G$ such that 
\begin{enumerate}
\item[$i)$] the $N_0$-neighbourhood of $\gamma$ is contained in $P^{-1}$ i.e. $\ball_X(\gamma(i), N_0)\subseteq P^{-1}\, \forall i\in \mathbb{Z}$,
\item[$ii)$] $\gamma(i)$ converges, as $i$ tends to $\pm\infty$, to two different boundary points of $\partial G$.
\end{enumerate}

We let $g_n=\max_{\prec}\ball_X(1_G, n)$. 
By  Proposition \ref{prop tool}, $n\mapsto g_n^{-1}$ is a geodesic ray such that $\ball(g_n^{-1},n)\subseteq P^{-1}$. 
We let $\xi=\lim g_n^{-1}\in \Lambda(G)$. 

Given any $g\in G$ there exist some generator $x\in X^{\pm 1}$ such that $gx$ does not belong to the stabilizer of $\xi$. Indeed, otherwise $\langle gX^{\pm  1} \rangle\supseteq \langle \{x_1x_2\mid x_1,x_2\in X^{\pm 1}\}  \rangle$ is a finite index subgroup  of $G$ that stabilizes $\xi$, contradicting that the action of $G$ on itself is of general type (Remark \ref{rem: hyperbolic groups are of general type}). 
Take $h\in g_{2N_0+1}^{-1}X^{\pm 1}\subseteq P^{-1}$ not belonging to the stabilizer of $\xi$ and define $\gamma:\Z\mapsto G$ as:
$$\gamma(n)=g_{2N_0+1+n}^{-1}\text{ if }n\geq 0;\qquad\qquad \gamma(n)=hg_{-(n+1)}^{-1}\text{ if }n<0.$$

We claim that $\gamma$ is the desired path.  
Notice that $\lim_{n\to+\infty}\gamma(n)=\xi$, $\lim_{n\to-\infty}\gamma(n)=h\xi$ and $h\xi$ differs from $\xi$ by construction of $h$. 
Thus $ii)$ holds. 

To see condition $i)$, take $n\in\Z$. 
If $n\geq 0$ the condition follows from the definition of $g_n$ and Proposition \ref{prop tool}. 
If $n\in[-N_0,-1]$ we get that $\gamma(n)\in h\ball(1_G, N_0)\subseteq \ball(g_{2N_0+1}^{-1},N_0+1)$ and therefore $\gamma(n)\ball(1_G,N_0)\subseteq \ball(g_{2N_0+1}^{-1},2N_0+1)\subseteq  P^{-1}$. 
Finally, if $n<-N_0$ we get $\gamma(n)=hg_{m}^{-1}$ for some $m\geq N_0$ and therefore $\gamma(n)\ball(1_G,N_0)=hg_{m}^{-1}\ball(1_G,N_0)\subseteq P^{-1}$ since $h\in P^{-1}$ and $g_{m}^{-1}\ball(1_G,N_0)\subseteq g_{m}^{-1}\ball(1_G,m)\subseteq P^{-1}$. 
\end{proof}

\subsection{Quasi-geodesic combings  of  positive cones}
\label{subsec combing}
Let $G$ be a non-elementary hyperbolic group and $X$ a finite generating set.
Let $\Gamma=\Gamma(G,X)$ be the Cayley graph.
A {\it $(\lambda, c)$-quasi-geodesic combing} of $P$ is a subset $\cP$ of $(\lambda,c)$-quasi-geodesics paths in $\Gamma$ starting at $1_G$ and ending at vertices of $P$, and such that for each $g\in P$, there is at least a path in $\cP$ from $1_G$ to $g$.
We say that $\cP$ is supported on the $r$-neighbourhood of $P$ if every $p\in \cP$ lies in  $\cup_{g\in P} \ball_X(g,r)$.

\begin{thm}\label{thm combing}
Let $G$ be a finitely generated and non-elementary hyperbolic group and $P$ a positive cone of $G$. Then, there is no quasi-geodesic combing of $P$ supported on a neighbourhood of $P$.

\end{thm}
\begin{proof}
Let $X$ be a finitely generating set of $G$ and $\Gamma=\Gamma(G,X)$ its associated Cayley graph.
Let $P$ be a positive cone of $G$ and suppose that there exists $\lambda \geq 1$, $c\geq 0$ and $r\geq 0$ and a $(\lambda,c)$-quasi-geodesic combing $\cP$ of $P$ supported on the $r$-neighbourhood of $P$.
Let $\prec$ denote the associated order of $P$.

Let $g_n=\max_\prec\mathbb{B}(1_G,n)$. 
By Proposition \ref{prop tool}, $n\mapsto g_n^{-1}$ is a geodesic ray. 
We denote by $\xi\in \partial G$ its limit as $n\to \infty$. 
Also by Proposition \ref{prop tool}, the horoball $\cup_n \ball(g_n^{-1},n-1)$ is contained in $P^{-1}$.

By Proposition \ref{prop: main boundary}, we know that there is a sequence $h_n\in P$ such that $h_n$ converges to $\xi$, {\em i.e.}  $(h_n|\xi)_{1_G}\to \infty$. 
Thus, for every $N>0$,  we can find $n(N),m(N)$ such that $(h_{n(N)}|g^{-1}_{m(N)})_{1}>N$.
Let $\alpha^{N}$ and $\beta^N$ be geodesic paths from $1_G$ to $h_{n(N)}$ and $g_{m(N)}^{-1}$ respectively.
By Proposition \ref{prop tool}, we can assume that $\beta^N(t)=g_t^{-1}$. 
By equation \eqref{eq: gromov product measures cancellation} in the beginning of Section \ref{sec: hyperbolic}, $\dist(\alpha^N(t),\beta^N(t))\leq 2\delta$ for all $t\in [0,N]$.

Denote by $p_n$ a path in $\cP$ from $1_G$ to $h_n$.
By the stability of quasi-geodesics, there is constant $D$ (only depending on $\delta$, $\lambda$ and $C$) such that $p_n$ and $\alpha_n$ are at Hausdorff distance at most $D$.
Suppose that $N> D+r+2\delta$.
Since $p_n$ is contained in the $r$-neighbourhood of $P$,  there is $z\in P$ such that $\dist(z,\alpha^N(N))\leq D+r$.
Since $\dist(\alpha^N(N), g_N^{-1})\leq 2\delta$, we have that $z\in \ball(g_N^{-1}, D+r+2\delta)\subseteq P^{-1}$, which is a contradiction.
\end{proof}

%
%
\section{Connectedness properties for positive cones}
\label{sec: prieto and hucha}
In this section we introduce two properties that essentially say that all positive cones have the same connectedness behaviour: either all positive cones are coarsely connected (Prieto property) or all positive cones are not coarsely connected (Hucha property).


\subsection{The Hucha property}
%


Let $\ga$ be a graph.
A subset $S$ of $V\ga$ {\it $r$-disconnects} subsets $H_1$ and $H_2$ of $\ga$ if any $r$-path from $H_1$ to $H_2$ has non-trivial intersection with $S$. We say that $S$ $r$-disconnects a subset $P$, if $P$ contains a subset $\{u,v\}$ so that $S$ $r$-disconnects $\{u\}$ and $\{v\}$.

\begin{defi}\label{defi:swamp}
Let $P$ be a positive cone of a group $G$, $H\leqslant G$ a subgroup and $\ga(G,X)$ a Cayley graph. 
A subset $S \subseteq P^{-1}$ is a {\it negative swamp of width $r$ for $H$} in $\ga(G,X)$ if: 
\begin{enumerate} 
\item $S$ $r$-disconnects $P$,
\item there are $g_1,g_2\in G$ such that $S$ $r$-disconnects $g_1H$ and $g_2H$.
\end{enumerate}
\end{defi}

\begin{rem}\label{rem: flexibility swamp}
If $S\subseteq P^{-1}$ is a negative swamp of with $r$ for $H$, then the same holds for any $S'$ with $S \subseteq S'\subseteq P^{-1}$. That is, $S'$ is also a negative swamp of width $r$ for $H$.

If $S$ is a negative swamp of with $r$ for $H$, then the same holds for any subgroup $K$ of $H$. That is $S$ is also a negative swamp of width $r$ for $K$.
\end{rem}

The Definition \ref{defi:swamp} is modeled after our proofs. We could just defined what it means for $P^{-1}$ to be a negative swamp. However, when we establish that $P^{-1}$ is a negative swamp, in many case we use a proper subset of $P^{-1}$ to show that it disconnects the appropriate sets. For example, we just used balls in the argument of Corollary \ref{cor: free groups} in the introduction. 

\begin{defi}\label{def: Hucha} 
Let $G$ be a left-orderable group, $X$ a finite generating set and $\mathcal{H}$ a family of subgroups. 
We say that the Cayley graph $\ga(G,X)$ is {\it Hucha with respect to $\mathcal{H}$} if for every positive cone $P$ of $G$, for every $H\in \mathcal{H}$ and for every $r>0$, there is is a negative swamp $S$ of width $r$ for $H$.

If $\mathcal{H}$ is the trivial subgroup we just say that $\ga(G,X)$ is Hucha.
\end{defi}

\begin{rem}\label{rem: family closed}
By Remark \ref{rem: flexibility swamp}, if $\ga(G,X)$ is Hucha with respect to $\cH$, then it is also Hucha with respect to $\{K\leqslant H \mid H\in \cH\}$. Thus  one might assume that the family $\mathcal{H}$ is closed under taking subgroups.
\end{rem}

By Remark \ref{rem: family closed}, if $\ga(G,X)$ is Hucha with respect to $\mathcal H$ then $\ga(G,X)$ is Hucha. 

The following is worth noticing.
\begin{prop}\label{prop: Hucha are not coarsely connected}
If $\ga(G,X)$ is Hucha, then no positive cone $P$ of $G$ is  coarsely connected.
\end{prop}
\begin{proof}
Let $P$ be a positive cone and $r$ a positive integer. Since $\ga(G,X)$ is Hucha, there is a negative swamp $S\subseteq P^{-1}$ that $r$-disconnects $P$. In particular, there are $g_1,g_2\in P$ such that every $r$-path connecting them must go through $S$ and thus it can not be supported in $P$. Since $r$ is arbitrary we deduce that $P$ is not coarsely connected.
\end{proof}

The following is easy from the definitions. 
\begin{lem}\label{lem: Hucha is geometric}
Being Hucha with respect to $\mathcal H$ is independent of the generating set. 
\end{lem}
Therefore, we will say that a group $G$ {\it is Hucha with respect to $\mathcal H$} if some (any) Cayley graph is Hucha with respect to $\cH$.
By Remark \ref{rem: fg pos cone are coarsely connected} we also have the following:
\begin{cor}
Hucha groups do not have finitely generated positive cones.
\end{cor}

Although we will not need this, it is worth recording it.
 The proof is easy but will exemplify some of the ideas used later.
\begin{lem}\label{lem: Hucha is preserved by overgroups}
Let $G$ be a finitely generated left-orderable group, $H$ a finite index subgroup such that $H$ is Hucha with respect to a family $\cH$ of subgroups of $H$. Then $G$ is Hucha with respect to the family $\cH$.
\end{lem}
\begin{proof}
Fix a finite generating set $X$ of $G$. 
By Lemma \ref{lem: Hucha is geometric}, we can assume that $Y=X\cap H$ generates $H$.
Thus $\ga(H,Y)$ is a subgraph of $\ga(G,X)$.
Let $R$ such that $H\ball(1,R)=G$.
Fix a positive cone $P$ of $G$ with corresponding order $\prec$.
Let $t^{-1}=\max_\prec\ball(1,R)$.
Thus $t\ball(1,R)\subseteq P^{-1}\cup \{1\}$.

Let $r>0$ and $K\in \cH$. We need to find a negative swamp $S\subseteq P^{-1}$ for $K$ of width $r$.
Notice that $P'=t^{-1}Pt\cap H$ is a positive cone for $H$.
By hypothesis, for every $r'>0$ there exists $S'\subseteq (P')^{-1}$ and $a_1,a_2\in H$ such that $S'$ $r'$-disconnects $a_1K$ and $a_2K$ in $\ga(H,Y)$.
We take $r'$ satisfying that $|h|_Y\leq r'$ for every $h\in H$ with $|h|_X\leq 2R+r$.

Note that $S'\subseteq t^{-1} P^{-1} t$. Thus $tS't^{-1}\subseteq P^{-1}$ and therefore $tS'\ball(1,R)\subseteq P^{-1}$.
We claim that $tS'\ball(1,R)$ is a negative swamp of width $r$ for $K$.
Let $\{g_i\}_{i=0}^m$ be an $r$-sequence in $G$ with $g_0\in ta_1K$ and $g_m\in ta_2K$.
Let $h_0,h_m\in H$ such that $th_0=g_0$, $th_m=g_m$ and  for $0<i<m$, let $h_i\in H$ such that $\dist(g_i,th_i)\leq R$.
Then $\{h_i\}_{i=0}^m$ is an $r'$-sequence in $\ga(H,Y)$ from $a_1K$ to $a_2K$ and hence there is some $i$ such that $h_i\in S'$. 
Thus $th_i\in tS'$ and $g_i\in th_i\ball(1,R)$.

It remains to show that $tS'\ball(1,R)$ $r$-disconnects $P$.
By hypothesis, there are $a_1,a_2\in P'$ such that $S'$ $r'$-disconnects them.
Note that $ta_it^{-1}\in P$.
Let $\{g_i\}_{i=0}^m$ be an $r$-sequence from $ta_1t^{-1}$ to $ta_2t^{-1}$.
Let $h_0=a_1$, $h_m=a_2$, and for $1<i<m$ let $h_i\in H$ such that $\dist(g_i,th_i)\leq R$.
Arguing as before, since the $r'$-sequence $\{h_i\}_{i=0}^m$ meets $S'$, we conclude that $\{g_i\}_{i=0}^m$ meets $tS'\ball(1,R)$.
\end{proof}

We do not know if the Hucha property passes to finite index subgroups since it is not always the case that a left-order on a finite index subgroup can be extended to a left-order on the ambient group.
We remark that in all the groups that we show that are Hucha, the property is inherited by its finite index subgroups.

However, we know that the Hucha property is not preserved under direct products.
\begin{ex}\label{ex F2xF2}
The free group of rank 2, $F_2$ has the Hucha property with respect to the trivial group. This follows form Corollary \ref{cor: free groups}. 
It is known that the BNS-invariant $\Sigma^1(F_2\times F_2)$ is non-empty. 
Concretely, the map that sends each standard generator of $F_2\times F_2$ to $1\in\mathbb{Z}$ has finitely generated kernel. Therefore, by Proposition \ref{prop:BNS}, $F_2\times F_2$ has a coarsely connected positive cone.
\end{ex} 

\subsection{The Prieto property}\label{sec: Prieto}
We introduce a strong negation of the Hucha property.

\begin{defi}
Let $G$ be a left-orderable group and $X$ a finite generating set of $G$.
We say that $\ga(G,X)$ {\it is Prieto} if every positive cone $P$ of $G$ is coarsely connected.
\end{defi}

It is easy to see that Prieto is a geometric property of the group, and thus if a Cayley graph of $G$ is Prieto, then every finitely generated Cayley graph of $G$ is Prieto. We record this fact in the following lemma.

\begin{lem}\label{lem: Prieto is geometric}
Being Prieto is independent of the finite generating set.
\end{lem}

We say that $G$ {\it is Prieto} if there is some finite generating set $Y$ of $G$ such that $\ga(G,Y)$ is Prieto. The following is the analogous to Lemma \ref{lem: Hucha is preserved by overgroups} in the Prieto case. We omit the proof, since it is straightforward.

\begin{lem}
Let $G$ be a finitely generated left-orderable group, and $H$ a finite index Prieto subgroup.
Then $G$ is Prieto.
\end{lem}




The easiest example of a Prieto groups are finitely generated torsion free-abelian groups. 

\begin{prop}\label{prop: abelian groups are Prieto}
Finitely generated free-abelian groups are Prieto.
\end{prop}
\begin{proof}
We view the Cayley graph $\ga$ of $\Z^n$ with respect to a basis as a lattice in $\mathbb{R}^n$. 
It is well known (see \cite[Section 1.2.1]{GOD} and references therein) that for every positive cone $P$ of $\Z^n$ there exists a hyperplane $\pi$ of $\mathbb{R}^n$ passing through the origin, such that $P$ consists on the lattice points in some connected component of $\mathbb{R}^n \setminus \pi$, and perhaps some points on $\pi$ (if $\pi$ contains lattice points).
Therefore $P$ is a connected subset of $\ga$.
\end{proof}

To exhibit more involved examples, recall that given $\prec$ a left-invariant order on a group $G$. A subgroup $H$ of $G$ is {\it cofinal} if for every $g\in G$, there are $h_1,h_2$ such that $h_1\prec g\prec h_2$.

\begin{lem}\label{lem: cofinal center makes P connected}
Let $(G,\prec)$ be a finitely generated left-ordered group. 
Suppose that there is a finitely generated, cofinal and central subgroup. 
Then $P_\prec=\{g\in G \mid 1\prec g\}$ is coarsely connected.
\end{lem}
\begin{proof}
Fix a finite generating set $X$ of $G$.
Let $g_0,g_m\in P_\prec$ and $\{g_i\}_{i=0}^m$ be a $1$-sequence connecting them.
Let  $Z$ be a finitely generated, cofinal and central subgroup of $G$. 
Since $Z$ is cofinal, there is $z\in Z$ such that $z\prec g_i$ for $0<i<m$.
Without loss of generality, assume $z\prec 1$.
Then $\{z^{-1}g_{i}\}_{i=0}^m$ is a $1$-sequence contained in $P$.
Since $Z$ if finitely generated, it is free-abelian and hence Prieto.
Thus there is some $r>0$, such that $P_\prec \cap Z$ (with the induced metric) is $r$-connected.
Thus there is an $r$-path $\{z_i\}_{i=0}^k$ supported in $P_\prec$ from $1$ to $z^{-1}$.
Now the concatenation $\{g_0 z_i\}_{i=0}^{k}$ with $\{z^{-1} g_i\}_{i=0}^m$ followed by $\{ z_ig_m\}_{i=k}^0$ is an $r$-sequence supported in $P_\prec$ from $g_0$ to $g_m$.
\end{proof}

\begin{rem}\label{rem: F_2x Z}
The group  $F_2\times \Z$ is not a Hucha group neither a Prieto group.
This will follow by considering different lexicographic orders. For showing that is not Hucha, we take the order on $F_2$ as the leading order in the lexicographic order. For showing that it is not Prieto, we consider the order on $\Z$ as the leading one.

 Indeed, from Corollary \ref{cor: free groups} we know that positive cones of $F_2$ are not coarsely connected, so, if we fix $P$ a positive cone in $F_2$, it is easy to check that 
$P'= (P\times \Z) \cup (\{1\}\times \mathbb{Z}_{\geq 1})\subseteq F_2\times \Z$ is positive cone for $F_2\times \Z$ which not coarsely connected either.

If instead one takes $P'=(F_2\times \Z_{\geq 1})\cup (P\times \{0\})$ this is a positive cone for $F_2\times \Z$ on which $\Z$ is co-final, and by Lemma \ref{lem: cofinal center makes P connected}, $P'$ is coarsely connected.
It is worth noticing that not only $F_2\times \Z$ has coarsely connected positive cones, but also it has finitely generated ones. This was recently showed by  H. L. Su \cite{Su}. 
\end{rem}

The braid group $B_n$ can be defined as the mapping class group of a punctured disk with $n$ punctures. E. Artin gave a presentation for $B_n$
$$B_n=\langle \sigma_1,\dots, \sigma_{n-1}\, \mid  (\sigma_i \sigma_{i+1} \sigma_{i} = \sigma_{i+1} \sigma_i \sigma_{i+1}\,: \, 1\leq i \leq n-2), (\sigma_i \sigma_j= \sigma_j \sigma_i \,:\, |i-j|>1)\rangle. $$

Iterated Torus knot groups are groups with presentation
$$T_{n_1,\dots,n_k} =\langle a_1,\dots, a_k \mid a_1^{n_1}=a_2^{n_2}=a_3^{n_3}=\dots =a_k^{n_k}\rangle.$$

\begin{cor}\label{cor: Braids are Prieto}
Braid groups, iterated Torus knot groups  and $\widetilde{T}=\langle a,b,c\mid a^2=b^3=c^7=abc\rangle $ are Prieto groups.

\end{cor}
\begin{proof}
The left-orderability of these groups is well-known. The orderabilty of $B_n$ was first proved by P. Dehornoy (see \cite[Section 1.2.6]{GOD}). The observation that $\widetilde{T}$ is left-orderable is attributed to W. Thurston and rediscovered by G. Bergman (see \cite[Discussion before Example 3.2.7]{GOD}), while the left-orderability of $T_{n_1,\dots, n_k}$ is discussed in \cite[End of Section 2.2.]{GOD}.

With the given presentation, the center of $B_n$ is generated by $(\sigma_1\sigma_2\dots\sigma_{n-1})^2$. 
A. Clay showed that this cyclic group is cofinal in every left-order of $B_n$. 
See \cite[Example 3.5.14.]{GOD} for a proof. 
By Lemma \ref{lem: cofinal center makes P connected} each positive cone is coarsely connected.

For iterated Torus knot groups the same holds: $z=a_1^{n_1}$ is always cofinal and central. 
The centrality is clear from the presentation. 
For the cofinality we argue as follows. 
Let $\preceq$ be a left-order of $T_{n_1,\ldots,n_k}$ and set $U=\{t\in T_{n_1,\ldots,n_k}\mid t\preceq z^m \text{ for some } m\in \Z\}$. 
We claim that $a_i U\subseteq U $ for all $a_i$. 
Indeed, let $t\in U$, $t\prec z^m$. If $a_i t\prec t\prec z^m$ then $a_i t \in U$ by definition. 
On the other hand if $t\prec a_i t$ then, inductively, we get that $t\prec a_i^{n_i} t=zt$, so $a_i t\prec z^{m+1}$.  
This implies that $U$ is invariant under every generator, so $U=T_{n_1\ldots,n_k}$. In particular $z$ is cofinal.

The  preceding argument also shows that $a^2\in \widetilde{T}$ is always (central and) cofinal. 
\end{proof}
Another family of examples of groups with the Prieto property are left-orderable groups for which every order can be described by a regular language. This family is described in detail in \cite{ARS}.

In contrast with the Hucha property, we know that the Prieto property does not pass to finite index subgroups. 
Indeed, the previous corollary shows that $B_3$ is Prieto while the pure braid group  $PB_3$ is a finite index subgroup isomorphic to $F_2\times \Z$, which,  by Remark \ref{rem: F_2x Z}, we know  that is not Prieto.

We finish this section by showing that, unlike Hucha property (see Example \ref{ex F2xF2}), the Prieto property is stable under direct products.

\begin{prop}
The direct product of Prieto groups is Prieto.
\end{prop}
\begin{proof}
Let $A$ and $B$ be two finitely generated, left-orderable, Prieto groups.
Fix a positive cone $P$  of $G=A\times B$.
Let $X_A$ and $X_B$ be finite generating sets for $A$ and $B$ respectively.
We consider $X= X_A \times \{1\} \cup \{1\} \times X_B$ as generating set for $G$.
Since $A$ and $B$ are Prieto, we assume that $P\cap A$ is an $r$-connected subset of $\Gamma(A, X_A)$. 
Similarly, we assume that $P\cap B$ is an $r$-connected subset of $\Gamma(B, X_B)$.
Note that for simplicity we identify $A$ with the subgroup  $A\times \{1\}$ and similarly $B$ with $\{1\}\times B$.

Let $g_1=(a_1,b_1)$ and $g_2=(a_2,b_2)$ be two positive elements.
We will connected them with an $r$-path in $\Gamma(G,X)$ supported in $P$.

We can assume without loss of generality that  $(a_1,1_B)$ is in  $P$.
Indeed, if that was not the case, then $a_1^{-1}\in P \cup \{1\}$ and there is an $r$-path $\{h_i\}_{i=1}^n$ from $1_A$ to $a^{-1}$ in $\Gamma(A, X_A)$ supported on $P\cap A$. Thus $\{(a_1h_i, b_1)\}_{i=1}^n$ is an $r$-path in $P$ from $g_1$ to $(1_A,b_1)$. Let $a'\in P\cap A$, $|a'|_X\leq r$. Then there is an $r$-path from $g_1$ to $(a',b_1)$ supported in $P$ and we can replace $g_1$ by $(a',b_1)$ if necessary. 

Repeating the argument, we assume that $a_1,a_2\in P\cap A$ and $b_1,b_2 \in B\cap P$. 
Thus there are $r$-paths $\{y_i\}_{i=1}^m$ in $\Gamma(A, X_A)$ from $a_1$ to $a_2$ and $\{z_i\}_{i=1}^l$ in $\Gamma(B,X_B)$ from $b_1$ to $b_2$ both supported on positive elements.
Thus $\{(y_i, b_1)\}_{i=1}^m$ concatenated with $\{(a_2, z_i)\}_{i=1}^l$ is an $r$-path from $g_1$ to $g_2$.
\end{proof}

\section{Groups acting on trees}
\label{sec: trees}
In this section we will show the main technical result of the paper, which is a combination theorem that produces Hucha groups.
For that, we will first prove Proposition \ref{prop: P accumulates in ends of T}, which is a self contained proof of Proposition \ref{prop: main boundary} in the special case of groups acting on trees. 
There are several reasons for giving this proof. Firstly, it is not lengthy and  we produce a statement better adapted for our purposes, sparing some translation effort.
Secondly, to show that certain groups are Hucha with respect to cyclic subgroups, we are using something slightly more general than what was proved on Proposition \ref{prop: main boundary}, namely that for any loxodromic isometry $h$ of the tree, and any neighbourhood $A$ of an end, there is a coset of $\langle h\rangle$ whose limit points lie in $A$.

We clarify that throughout this paper trees are simplicial trees, and actions on trees are by graph automorphisms that do not reverse edge orientations, which are the standard conventions when dealing with Bass-Serre theory.

The key of this section is the strong interplay between the geometries of $T$ and of $G$, that arises from the action of $G$ on $T$. We could say that the geometry of the tree dominates the geometry of the group. This allows us to construct negative swamps and to show  that certain groups acting on trees are Hucha in Subsection \ref{ssec: creating hucha}.

\subsection{Positive Cones of groups acting on trees}
\label{sec: positive cones accumulate on the bountary of a tree}

Let $G$ be a group acting on a  tree $T$. We recall the following classical definitions and well known facts about such an action.
An element $g\in G$ that stabilizes a vertex of $T$ is called {\it $T$-elliptic}, and if $g$ is not $T$-elliptic  then it is called {\it $T$-loxodromic}.
If $g$ is a $T$-loxodromic element, then there exists a unique subtree of $T$ homeomorphic to $\mathbb{R}$ on which $g$ acts by translation and we denote it by $\axis_T(g)$, see for example \cite[Proposition I.4.11]{DicksDunwoody}. Then $\axis_T(g)$ is the minimal $\langle g \rangle$-invariant subtree of $T$. We will drop the $T$ from the notation when the action is clear from the context. We say that the action is \emph{minimal} when $T$ is the only $G$-invariant sub-tree. 

A reduced path in $T$ is a {\em geodesic} with respect to the combinatorial metric on $T$, that assigns length $1$ to the edges. Given $v_1,v_2\in T$ we denote $[v_1,v_2]_T$ the geodesic segment between them. 

When a path in $T$ is semi-infinite and reduced we will call it a \emph{ray}. We say that two rays of $T$ are {\em equivalent} if they have infinite intersection, and define an {\em end} of $T$ as an equivalence class of rays. It is easy to see a correspondence between the ends of $T$ and the Gromov boundary $\partial T$ that comes form the combinatorial metric on $T$, though in this section we prefer to take the more specific approach of focusing on ends. We say that $G$ {\it fixes an end of $T$} if there is a ray $p$ in $T$ such that
$g p \cap p$ is infinite for all $g\in G$.

Recall from Section \ref{sec: hyperbolic} that the action of a group $G$ on a tree $T$ is {\it of general type} if $G$ does not fix neither a vertex nor an end of $T$, and the $G$-orbit of a vertex accumulates on at least 3 different ends of $T$. We see from Theorem \ref{thm: classification actions} that a general type action on a tree always has loxodromic elements. A direct proof of this fact can be found in \cite[Theorem I.4.12]{DicksDunwoody}.


\begin{ex}\label{rem: example}
For the purposes of the paper, the main examples of general type  actions on trees are the following: Let  $G$ be either an amalgamated free product $A*_C B$, with $C$ a proper subgroup of $A$ and $B$, or an HNN extension $A*_C t$, with $C$ a proper subgroup of $A$.
Then the action of $G$ on its associated Bass-Serre tree is of general type.  
If $G=A*_C B$, the Bass-Serre tree has vertex set $G/A \sqcup G/B$ and edge set $G/C$ with $gC$ adjacent to $gA$ and $gB$.  
If $G= A*_C t$, the Bass-Serre tree has vertex set $G/A$ and edge set $G/C$ with $gC$ adjacent to $gA$ and $gtA$. 
The action of $G$ on the Bass-Serre tree is induced by the left multiplication action of $G$ on its subgroup's cosets. See \cite{serre} for a detailed discussion.
\end{ex}

The following fact is well-known, we include the proof for completeness.
\begin{lem}\label{lem:h_1h_2}
Let $G\curvearrowright T$ be a general type, co-bounded action of a group $G$  on a tree $T$.
Let $h$ be an element acting loxodromically on $T$.
Then there are $h_1,h_2\in h^G$,  two $G$-conjugates of $h$, such that 
$\axis_T(h_1)\cap \axis_T(h_2)$ is empty. 
\end{lem}

\begin{proof}
Recall that for $g\in G$, one has that  $g\axis_T(h)=\axis_T(ghg^{-1})$.
Since $T$ has more than two ends and the action of $G$ is co-bounded, we see that there is some $g\in G$ such that $g\axis_T(h)\neq \axis_T(h)$. 
Moreover, since $G$ does not fix an end of $T$,
there must be $g_1$ and $g_2$ in $G$ such that
$g_1\axis_T(h)\cap g_2 \axis_T(h) \cap \axis_T(h)$ is finite  (i.e. that $g_1\axis_T(h)$ and $g_2\axis_T(h)$  could only have infinite intersection with different ends of $\axis_T(h)$).
The action of $\gen{h}$ on $\axis_T(g_1hg_1^{-1})$ ``slides'' this axis along $\axis_T(h)$, therefore we can find $n$ such that
$h^n\axis_T(g_1hg_1^{-1})\cap \axis_T(g_2hg_2^{-1})$ is empty. 
Now we can take $h_1=h^ng_1hg_1^{-1}h^{-n}$ and $h_2=g_2hg_2^{-1}$.
\end{proof}

We are ready to show that the orbit of a positive cone visits every neighbourhood of every end of $T$.

\begin{prop}\label{prop: P accumulates in ends of T}
Let $G\curvearrowright T$ be a co-bounded, general type action of a group $G$ on a tree $T$.
Let $h\in G$ be an element acting loxodromically on $T$, and $v$ a vertex of $T$.
 Consider $B$ a bounded subset of $VT$ and $\cC$ an unbounded connected component of $T-B$. 

Then there exists $k\in G$ such that  $k\langle h \rangle v =\{kh^n v \mid n\in \Z\}$ lies in $\cC$. Moreover, for every positive cone $P$ of $G$ the set $Pv$ has non-trivial intersection with every unbounded connected component of $T-B$.
\end{prop}

\begin{proof} 
By Lemma \ref{lem:h_1h_2} there exist loxodromic elements  $h_1=g_1hg_1^{-1}$ and $h_2=g_2hg_2^{-1}$ such that 
$\axis_T(h_1)\cap \axis_T(h_2)=\emptyset$.
Let $\mathcal{O}_i$ be the orbit of $g_iv$ under $\langle h_i\rangle$, that is $\mathcal{O}_i = \{g_ih^n v \mid n\in \Z\} =g_i\langle h \rangle v$. 
Let $L$ denote the minimal segment connecting $\mathcal{O}_1$ and $\mathcal{O}_2$. 
 Since the action of $G$ on $T$  is co-bounded, there is some $g\in G$ such that $g L\subseteq \mathcal{C} $.
Moreover, we can assume that 
\begin{equation}\label{eq:sufficiently far away L}
\dist(g L, B)>D. 
\end{equation}
where $D$ is an arbitrary constant, which we will choose later.

We claim that for a choice of $g$ that makes $D$  sufficiently large,  either $g\mathcal{O}_1$ or $g\mathcal{O}_2$ is contained in $\mathcal{C}$. 
Suppose this is not the case, thus for $i=1,2$ there are $w_i\in  B$ that lie in the geodesic between two consecutive vertices of $g\mathcal{O}_i$. To be more specific, for each $i=1,2$ there is  some $n$ such that $gg_ih^nv$ and $gg_ih^{n+1}v$ are in different components of $T-B$, thus the geodesic segment $[gg_ih^nv,gg_ih^{n+1}v]_T$ must meet $B$ in at least a point $w_i$. 
Note that $\dist(gg_ih^nv,gg_ih^{n+1}v)=\dist(hv,v)$ for every $n\in\Z$, since $G$ acts by isometries of the combinatorial metric. 
Therefore  $$\dist(g\mathcal{O}_1,g\mathcal{O}_2)\leq \dist(w_1,w_2)+2\dist(v,hv)\leq \diam(B)+2\dist(v,hv).$$
On the other hand, we get from \eqref{eq:sufficiently far away L} that $\dist(w_1,w_2)\geq 2 D+ \diam (L)-2 \dist(v,hv)$.
Putting these two things together we get that $2D +\diam(L)\leq \diam(B)+2\dist(v,hv)$. Since we can make $D$ arbitrarily big, for an appropriate choice of $g$ we get a contradiction.
Thus either $g\mathcal{O}_1$ or $g\mathcal{O}_2$ in $\cC$ and this completes the proof of the first statement.

For the second statement, let $P$ be a positive cone and $\cC$ an infinite component of $T-B$. Pick $k\in G$ so that $k\langle h \rangle v\subset \cC$. We consider the geodesic segment $[v,kv]_T$ and act on it by $\langle khk^{-1} \rangle$, obtaining the segments $k h^{n} k^{-1}[v,kv]_T = [k h^{n} k^{-1} v, kh^n v]_T$. Since $kh^n v \in\cC$ for all $n$, the set $B$ is bounded and $khk^{-1}$ is loxodromic, we see that all but finitely many of these segments lie inside $\cC$. In particular, $k h^{n} k^{-1} v\in\cC$ for all but finitely many $n\in\Z$. 

Since either $k h k^{-1}\in P$ or $k h^{-1} k^{-1}\in P$,
we get that $\langle khk^{-1} \rangle v$ contains infinitely many elements of $Pv$. 
\end{proof}

The following is an strengthening of Corollary \ref{cor: free groups}, and it will be used as the base case for showing that limit groups are Hucha.

\begin{prop}\label{prop:F2hucha} 
Finitely generated non-abelian free groups are Hucha with respect to the family of finitely generated, infinite-index subgroups.
\end{prop}
\begin{proof}
Let $F$ be a finitely generated non-abelian free group and $X$ a finite generating set. 
Without loss of generality, by Lemma \ref{lem: Hucha is geometric}, we can assume that $X$ is a basis and hence the Cayley graph of $F$ is a tree $T$ with infinitely many ends.
The action of $F$ on $T$ is co-compact, free and of general type.

Fix a positive cone $P$, an infinite-index finitely generated subgroup $H$ and  $r\geq  0$. 
Denote by $\prec$ the order on $F$ corresponding to $P$. 
Let $x_n=\max_{\prec}\ball(1_F, n)$ and recall that $x_n^{-1}\ball(1_F,n-1)\subseteq P^{-1}$.
Take $S= x_{r+1}^{-1}\ball(1_F, r)$. 
Since $T$ is a tree, there are two different components of $T-S$ so that every $r$-path joining them must intersect $S$. 
To be more specific, this happens when the geodesic joining these two complementary components passes through the ball's center $x_{r+1}^{-1}$. On the other hand, by Proposition \ref{prop: P accumulates in ends of T} every connected component of $T-S$ contains positive elements. Thus $P$ is not $r$-connected.

Every finitely generated subgroup of a free group is {\em quasi-convex} (geodesically, not with respect to the order $\prec$), therefore by  \cite[Theorem 4.8]{HruskaWise} we get that $H$ has {\it bounded packing}, meaning that for every $D$ there is some number $N=N(D)$ such that for any collection $g_1H,\dots, g_N H$ of $N$ distinct cosets of $H$, there are at least two of them separated by a distance of at least $D$.
 
So there exist cosets $g_1 H$ and $g_2 H$ at distance $>2r$.
Let $\gamma$ be a geodesic path joining $g_1 H$ and $g_2 H$ and assume that $v$ is a vertex in $\gamma$ with $\dist(v, g_1 H)>r$ and $\dist (v,g_2H)>r$. 
Since $F$ acts transitively by isometries, we can assume that $v=x_{r+1}^{-1}$ and therefore, $g_1H$ and $g_2H$ are separated by $S=x_{r+1}^{-1}\ball(1_F,r)$ where the geodesic connecting them passes through $x_{r+1}^{-1}$ as we discussed before. Thus there is no $r$-path in $P$ connecting $g_1 H$ and $g_2 H$.
 \end{proof}

\subsection{Cayley graphs of groups acting on trees}
\label{ssec: T dominates}

Consider a group $G$ acting on a a tree $T$. In this subsection we will show how to find generating sets for $G$ which are  \emph{adapted} for the action on $T$ (Lemma \ref{lem:adapted}). 
These generating sets will be useful to relate the geometries of $G$ and $T$. 

%
%

For $v\in VT$ denote $\text{link}_T(v)$ the set of edges in $ET$ which are adjacent to $v$. 
We say that a sub-tree $T_0\subseteq T$ is a \emph{fundamental domain} for the $G$-action on $T$ if $T_0$ contains exactly one representative of every edge-orbit under the action. 
Also, given $v\in VT$ we denote $G_v$ its stabilizer under the $G$-action, that is $G_v=\{g\in G:gv=v\}$. 
Analogously, for $e\in ET$, denote the edge stabilizer by $G_e$. 

The following is the key definition for our discussion of adapted generating sets. Given $X$ a generating set for $G$, $v\in VT$ and $E\subseteq \text{link}_T(v)$, we say that $X$ {\em has $v$-reductions modulo $E$} if for every $1$-path $\gamma\colon\{0,\ldots,n\}\to\ga(G,X)$ with \begin{itemize}
\item $\gamma(i)v\neq v$ for $0<i<n$, and
\item $\gamma(0)v=\gamma(n)v=v$,
\end{itemize} we have that $\gamma(n)\in\gamma(0)G_e$ for some $e\in E$. 

In other words, if a generating set $X$ has $v$-reductions modulo $E$, then for every path in $\Gamma(G,X)$ leaving and then coming back to $G_v$ there is and edge $e\in E \subseteq \text{link}_T(v)$ and an element $g\in G_v$ such that the path leaves from and comes back to (the same) coset $gG_e$.  Informally, this is saying that cosets of edge stabilizers are check-points for paths in $\Gamma(G,X)$ passing through cosets of vertex stabilizers.

Next we show the existence of generators having $v$-reductions, as well as some other properties. 
The statement of the following lemma is implicit in Serre's book \cite{serre} but we include a full proof for completeness.

\begin{lem}\label{lem:adapted} 
Let $G$ be a finitely generated group with a co-compact action on a tree $T$ with finitely generated vertex stabilizers. 
Given a vertex $v\in VT$,  there exists a compact fundamental domain $T_0$ containing $v$ and a symmetric and finite generating set $X$ satisfying:

\begin{enumerate}
\item[(1)] $X\cap G_v$ generates $G_v$,
\item[(2)] $Gv\cap [v,sv]_T=\{v,sv\}$ for every $s\in X$,
\item[(3)] $s T_0\cap T_0$ is non-empty for every $s\in X$,
\item[(4)] There exists a finite subset $E\subseteq \text{link}_T(v)$ so that $X$ has $v$-reduction modulo $E$. 
\end{enumerate}
\end{lem}

\begin{proof} 
First we construct the fundamental domain $T_0$ and the generating set $X$. After that we shall verify the properties (1)--(4).
\paragraph{Construction of $T_0$:} 
Let $A:=G \backslash T$ be the quotient graph of $T$ by the action of $G$. Let $\pi:T\to A$ be the associated quotient map,  and consider $A_0\subseteq A$ a maximal subtree  of $A$. 
Let $T_0'$ be a lift of $A_0$ that contains $v$, and notice that it is a subtree $T_0'\subseteq T$ so that $\pi$ induces an isomorphism between $T_0'$ and $A_0$. 

Finally we extend $T_0'$ to a fundamental domain $T_0$ such that every edge in $ET_0\setminus ET_0'$ has one endpoint in $T_0'$. We do so by choosing lifts of the edges $\bar{e}\in EA\setminus EA_0$ that begin in a vertex of $T'_0$. There are several choices on how to do this, and we can require one extra condition: that an edge $e\in ET_0$ is adjacent to $v$ whenever $\pi(e)$ is adjacent $\pi(v)$, i.e. the edges $e\in ET_0$ are adjacent to $v$ when that is possible. This is achieved by choosing a lift that begins at $v$ for every $\bar{e}\in EA\setminus EA_0$ with $\pi(v)$ as an endpoint. 

Note that $T_0$ is compact because, by hypothesis, the action is co-compact. It will also be important to notice that our construction yields that $$E^{\ast}:=\text{link}_{T_0}(v) $$ is a set of representatives for the orbits of edges that have some endpoint in the orbit of $v$.

\paragraph{Construction of $X$:} 

For every $w\in VT_0'$ consider a finite and symmetric set $X_w$ so that $\langle X_w \rangle=G_w$. 
We can do that because, by hypothesis, vertex stabilizers are finitely generated. 
Also, for every edge $e=[w_1,w_2]_T$ in $ET_0\setminus ET_0'$ with $w_1\in VT_0'$ we consider an element $g_e\in G$ so that $g_e^{-1}w_2\in T_0'$. 
We define $$X=(\cup_{w\in VT_0'}X_w)\bigcup \{g_{e},g_{e}^{-1}:e\in ET_0\setminus ET_0'\}.$$

\paragraph{Condition (1):} It follows immediately from our definition of $X$, recalling that $v\in VT'_0$.

\paragraph{Conditions (2) and (3):}

Take $s\in X$. We distinguish two cases.

\subparagraph{Case 1:} $s\in G_w$ for some $w\in VT_0'$. 

We write $[v,sv]_T\subseteq [v,w]_T\cup[w,sv]_T$ and notice that $[w,sv]_T = s[w,v]_T$. 
Since $T_0'$ meets every vertex orbit under $G$ exactly once and $[v,w]_T\subseteq T_0'$, we can deduce that $[v,sv]_T\cap Gv=\{v,sv\}$, which proves condition (2) in this case. 
Since $w\in T_0\cap sT_0$ condition (3) follows. 

\subparagraph{Case  2:} $s\notin G_w$ for every $w\in VT_0'$. 

Since the action is simplicial and $T_0'$ meets every vertex orbit exactly once, we see that if $sT_0'\cap T_0'\neq\emptyset$ then $s\in G_w$ for some $w\in VT_0'$. So we can assume that $sT_0'\cap T_0'=\emptyset$ for this case. 
We may assume without loss of generality that $s=g_e$ for some $e\in ET_0\setminus ET_0'$, since the argument for $s=g_e^{-1}$ is analogous.
Denote $e=[w_1,w_2]_T$ with $w_1\in T_0'$. 
By construction of $g_e$ we have that $w_2\in g_eT_0'$, 
thus $w_2\in T_0\cap g_eT_0$ which gives condition (3). 

To verify condition (2) note that, since $e$ meets both $T_0'$ and $g_eT_0'$, we have that 
$$[v,g_ev]_T\subseteq T_0'\cup e\cup g_eT_0'.$$

So $Gv\cap [v,g_ev]_T \subset VT'_0 \cup Vg_eT_0'$. On the other hand, the orbit of the vertex $v$ meets $VT'_0 \cup Vg_eT_0'$ exactly on $v$ and $g_ev$, by construction of $T'_0$. This gives condition (2) in this case.  


\paragraph{Condition (4):} 



We define the suitable set $E\subset\text{link}_T(v)$ by enlarging $E^{\ast} = \text{link}_{T_0}(v)$ as follows: $$ E = E^{\ast} \cup \{ g_e^{-1}e : e\in E^{\ast} \mbox{ such that } g_e^{-1}e\in  \text{link}_T(v) \}$$

By the choice of $T_0$, this amounts to adding  the edges of the form $g_e^{-1}e$ for $e\in ET_0\setminus ET'_0$ that have $v$ as an endpoint. The added edges are always lifts of loops in $A=G \backslash T$ based at $\pi(v)$. It is clear that $E$ is finite, since $E^{\ast}$ is. The reason for this construction will become clear at the very end of the proof.

We will show that $X$ has $v$-reduction modulo $E$. For this consider a $1$-path $\{\gamma_i\}_{i=0}^n$ in $\ga(G,X)$ so that $\gamma_0 v=\gamma_n v=v$ and $\gamma_iv\neq v$ for $0<i<n$. 
By condition (2) we have that $\{\gamma_i\}_{i=1}^{n-1}$ is supported on a single connected component of $T\setminus \{v\}$, whose closure (by adding $v$) we denote by $T_{\ast}$.   

Let $\overline{e}$ be the {\it unique } edge in $ET_{\ast}$ that is adjacent to $v$. Recalling that $E^{\ast}$ is a set of representatives of the edges with endpoints in $Gv$, we get that there is a unique $e\in E^{\ast}$ and some $g\in G$ with $\overline{e} = ge$.


In order to show that $\gamma$ $v$-reduces modulo $E$, we will distinguish two cases. Let $d_T$ be the combinatorial metric of the tree $T$.

\subparagraph{Case 1:} $\dist_T(v,\gamma_1 v)>1$. 

First we show that $\overline{e}=\gamma_0 e$. Since $\dist_T(v,\gamma_1 v)>1$ we see that $\gamma_1e$ is not adjacent to $v$, in particular $\gamma_1 e\neq \overline{e}$. Since $T_0$ is a fundamental domain, $\gamma_1e$ is the only edge in the orbit of $e$ that lies in $\gamma_1 T_0$, therefore $\overline{e}$ does not belong to $\gamma_1 T_0$. 
On the other hand, condition (3) gives that $\gamma_0 T_0\cup\gamma_1 T_0$ is connected and so it contains the geodesic segment $[v,\gamma_1 v]_T$. This segment is also contained in $T_{\ast}$, so its first edge is $\overline{e}$. 
Then we must have $\overline{e}\in\gamma_0 T_0$, which implies that $\overline{e}=\gamma_0 e$. 

Now we point out that $\dist_T(\gamma_{n-1}v,v)>1$. 
Indeed, if it was not the case we would have that $\gamma_{n-1}v$ would belong to the interior of $[\gamma_0 v,\gamma_1 v]_T$ contradicting condition (2). 
Since $\dist_T(\gamma_{n-1}v,v)>1$, we can repeat the previous argument and show that $\overline{e}=\gamma_n e$. 
This implies that $\gamma_0^{-1}\gamma_n\in G_{e}$.
Therefore, since $e\in E$, condition (4) holds in this case.  

\subparagraph{Case 2:} $\dist_T(v,\gamma_1 v)=1$.

Arguing as in the previous case we can use condition (2) to show that $\dist_T(\gamma_{n-1}v,v)=1$. 
Since $\{\gamma_i\}_{i=0}^n$ is a $1$-path in $\ga(G,X)$ there are $s_1,s_n\in X$ such that $\gamma_{1}=\gamma_{0}s_1$ and $\gamma_{n}=\gamma_{n-1}s_n$.  

We claim that $s_1,s_n\in\{g_e,g_{e}^{-1}\}$.  
 We prove it for $s_1$, since the other case is analogous. 
 For this, first notice that 
\begin{equation}\label{equdist}\dist_T(v,s_1v)=1.
\end{equation}
 This tells us that $s_1$ is not elliptic, because whenever $h\in G$ is elliptic we have that $\dist_T(v,hv)$ is even. Therefore $s_1=g_{e_0}^{\pm1}$ for some $e_0\in T_0\setminus T_0'$. Moreover, the edge $[v,s_1v]_T$ is either $e_0$ or $g_{e_0}^{-1}e_0$, and in any case we have $e_0\in E^{\ast}$. On the other hand $\overline{e} = [v,\gamma_1 v]_T = \gamma_0[v,s_1v]_T = ge_0$ for either $g=\gamma_0$ or $g=\gamma_0g_{e_0}^{-1}$. This gives us $e_0 = e$, and therefore the claim.
 
It will be useful to record that either $\overline{e}=\gamma_0e$ with $s_1=g_e$, or $\overline{e}=\gamma_1e$ with $s_1=g_e^{-1}$. The same argument would yield that $\overline{e}$ is either $\gamma_n e$ or $\gamma_{n-1}e$, according to whether $s_n$ is $g_e^{-1}$ or $g_e$.
 


Now, in order to finish the proof that condition (4) holds, we will distinguish in two sub-cases according to $s_1=g_e$ or $s_1=g_e^{-1}$. 

\subparagraph{Subcase 2A:} $s_1=g_e$. 

In this case we have that $\overline{e}=\gamma_0 e$. We argue that $\overline{e}=\gamma_n e$ by contradiction: Assuming the contrary we would have $\overline{e}=\gamma_{n-1} e$, but then $\gamma^{-1}_{n-1}\gamma_0$ would fix the edge $e$ without fixing the endpoint $v$. This is not possible, since the action has no edge inversions. Thus we get that $\overline{e}=\gamma_n e$ and so $\gamma_0^{-1}\gamma_n \in G_e$. Recall that $e\in E$, so this shows condition (4) in this instance.


\subparagraph{Subscase 2B:} $s_1=g_e^{-1}$. 

Now we have that $\overline{e}=\gamma_1e$, and we can show that $\overline{e}=\gamma_{n-1}e$ by the same argument used in the previous case: Supposing that $\overline{e}=\gamma_ne$ would yield that $\gamma_n\gamma^{-1}_1$ induces an edge inversion on $e$. As previously mentioned, this allows us to deduce that $s_n=g_e$. 

So we get that $\overline{e}=\gamma_{n-1}e=\gamma_{n}g_e^{-1}e$, and recalling that $\overline{e}=\gamma_{1}e=\gamma_{0}g_e^{-1}e$ we deduce that $\gamma_0^{-1}\gamma_n\in G_{g_e^{-1}e}$. 
Since $g_e^{-1}e\in E$ this finishes the proof of condition (4). 

 \end{proof}

If $T_0$ is a sub-tree of a tree $T$, and $v$ is a vertex of $T_0$, we call the pair $(T_0,v)$ a {\it pointed subtree} of $T$. 

\begin{defi}
Let $G$ be a group acting on a tree $T$ with the hypothesis of Lemma \ref{lem:adapted}.
If a generating set $X$ of $G$ satisfies the conclusions (1)--(4) of Lemma \ref{lem:adapted} for a pointed sub-tree $(T_0,v)$, we will say that $X$  is a \emph{adapted to} $(T_0,v)$.  Lemma \ref{lem:adapted} shows that adapted generators always exist.  
\end{defi}

\begin{rem} \label{rem:disconect}
If $X$ is a generating set of $G$ adapted to $(T_0,v)$, then $\ga(G,X)\setminus G_v$ is not connected. 
Indeed, if $g_1v$ and $g_2v$ are in different components of $T\setminus\{v\}$, then for every path $\{\gamma_i\}_{i=0}^n$ in $\Gamma(G,X)$ joining $g_1$ and $g_2$, by condition (2)  in Lemma \ref{lem:adapted} there exist $i$ such that $\gamma_iv=v$.
\end{rem}

\subsection{Creating Hucha groups}
\label{ssec: creating hucha}

Given a geodesic segment $p=[w_1,w_2]_T$ in $T$, denote by $T^p_{w_1}$ the minimal subtree of $T$ that contains $w_1$ and the connected components of $T\setminus\{w_1\}$ that do not contain $w_2$. Clearly $T^p_{w_2}$ is defined as well. Given $v\in VT$ let $\cC^p_{w_i}\subset G$ be the pre-image of $T^p_{w_i}$ under the orbit map of $v$. 
That is, for $i=1,2$, define $$\cC^p_{w_i}=\{g\in G:gv\in T^p_{w_i}\}.$$

The marked point $v$ is absent from the notation, as it will be clear from the context. We may occasionally drop the $p$ from the notation as well.
Our next immediate objective is to reduce the problem of proving the Hucha property with respect to cyclic subgroups, to showing that negative cones $r$-disconnect subsets of the form $\cC^p_{w_i}$. This is stated precisely in the following:

\begin{lem}\label{lem:disconecting} 
Let $G\curvearrowright T$ be a co-compact, minimal, general type action of a left-orderable group $G$ on a tree $T$.
Let  $v\in VT$ and $X$ be a generating set  of $G$. 
Assume that for every positive cone $P$ and any $r>0$ there exists a geodesic segment $p=[v_1,v_2]_T$ in $T$ and a set $S\subset P^{-1}$ so that $S$ $r$-disconnects $\cC^p_{v_1}$ from $\cC^p_{v_2}$ inside $\ga(G,X)$. 

Then $G$ has the Hucha property with respect to subgroups acting elliptically on $T$ and cyclic subgroups acting loxodromically.
In particular, $G$ has the Hucha property with respect to all cyclic subgroups. 
\end{lem}

\begin{proof} 

Let $P$ be a positive cone in $G$, $H$ be either an elliptic subgroup or a cyclic subgroup acting loxodromically, and $r>0$. In order to prove the desired Hucha property, we need to find a negative swamp in $P^{-1}$ of width $r$ for $H$. 
For this we consider the geodesic segment $p=[v_1,v_2]_T$ and the subset $S\subseteq P^{-1}$ provided by the hypothesis, so that $S$ $r$-disconnects $\cC_{v_1}$ from $\cC_{v_2}$ in $\Gamma(G,X)$. 

We aim to show that $S$ is our negative swamp. For this it is enough to show that:
\begin{enumerate}
\item $P\cap \cC_{v_i}$ is non-empty for $i=1,2$,
\item there are cosets $h_iH$ for $i=1,2$, such that $h_iH$ is contained in $\cC_{v_i}$. 
\end{enumerate}


Since $T$ is an infinite tree with no leaves, we see that $T_{v_i}$ are unbounded for $i=1,2$. 
On the other hand, the action is of general type, so there is some $h\in G$ acting loxodromically on $T$. Then we may apply Proposition \ref{prop: P accumulates in ends of T} to conclude that $Pv$ intersects both $T_{v_1}$ and $T_{v_2}$. This implies that $P\cap \cC_{v_i}$ is non-empty for $i=1,2$, proving point 1. 

We will split the proof of point 2 into cases, according to the action of $H$ on $T$. 


If $H$ is elliptic, there exists $w\in VT$ so that $H\subseteq G_w$. Let $d_T$ be the combinatorial metric on $T$, take $R=d_T(v,w)$ and fix $i=1,2$. Notice that since $T_{v_i}$ is unbounded and the action is co-compact, there is $h_i\in G$ so that $h_iw\in T_{v_i}$ and $d_T(h_iw,v_i) > R$. That is to say that the $R$-ball in $T$ around $h_iw$ is contained in $T_{w_i}$. Since $h_iHw=h_iw$ we get that $h_iHv$, whose points lie at distance $R$ from $h_iw$, is contained in $T_{v_i}$. Therefore $h_iH\subseteq \cC_{v_i}$ as desired. 


If $H$ is infinite cyclic acting loxodromically, we are in condition to apply Proposition \ref{prop: P accumulates in ends of T} to show that there exist $h_1,h_2\in G$ so that $h_i H v\subseteq T_{v_i}$ for $i=1,2$. 
This implies that $h_iH\subseteq \cC_{v_i}$ for $i=1,2$, proving point 2.
\end{proof}

We proceed to prove our combination theorems, starting with the easier case where the action has an edge with trivial stabilizer.
It will be useful to consider {\em geodesic interpolations} of $r$-paths in $\ga(G,X)$. Namely, given an $r$-path $\gamma\colon\{0,\ldots,n\}\to\ga(G,X)$ and a $1$-path $\overline{\gamma}\colon\{0,\ldots,m\}\to\ga(G,X)$, we say that $\overline{\gamma}$ is a \emph{geodesic interpolation} of $\gamma$ if there exists a monotone map $\sigma\colon\{0,\ldots,n\}\to\{0,\ldots,m\}$ satisfying:\begin{itemize} 
\item $\gamma=\overline{\gamma}\circ\sigma$
\item $\{\overline{\gamma}_j\}_{j=\sigma(i)}^{\sigma(i+1)}$ is a geodesic for $i=0,\ldots,n-1$.
\end{itemize}

\begin{prop}\label{prop: free product Hucha}
Suppose that $G\curvearrowright T$ is a co-compact, minimal, general type  action of a finitely generated left-orderable group $G$ on a tree $T$. 
Assume there exists $e\in ET$ with trivial stabilizer and that vertex stabilizers are finitely generated. Then $G$ has the Hucha property with respect to cyclic subgroups and subgroups acting elliptically on $T$. 
\end{prop}
\begin{proof}
By Lemma \ref{lem:adapted}, we can consider a finite generating set $X$ and a finite pointed sub-tree $(T_0,v)$ so that $X$ is adapted to $(T_0,v)$. 
Take $\prec$ a left-order on $G$ with positive cone $P$ and $r>0$. 
We will find a subset $S\subseteq P^{-1}$ and a geodesic segment $[w_1,w_2]_T$ so that $S$ $r$-disconnects $\cC_{w_1}$ from $\cC_{w_2}$ in $\ga(G,X)$. 
Then the proposition will follow from Lemma \ref{lem:disconecting}. 

Take $g\in G$ such that $\max_\prec\ball_X(1,r)\prec g^{-1}$, 
so we have that $g\ball_X(1,r)\subseteq P^{-1}$. 
Since $T_0$ is a fundamental domain we have $e_0\in ET_0$ with trivial stabilizer. We set $[w_1,w_2]_T\coloneqq ge_0$ and $S\coloneqq g\ball_X(1,r)$, and claim that $S$ $r$-disconnects $\cC_{w_1}$ from $\cC_{w_2}$. 

To see this, take $\gamma=\{\gamma_i\}_{i=0}^n$ an $r$-path between $h_1\in\C_{w_1}$ and $h_2\in \cC_{w_2}$, and consider $\overline{\gamma}=\{\overline{\gamma}_i\}_{i=0}^m$ a geodesic interpolation of $\gamma$. 
By condition (3) in Lemma \ref{lem:adapted} we have that $\overline{\gamma}_{i+1}T_0\cap\overline{\gamma}_iT_0\neq\emptyset$ for every $i$, 
thus $ge_0\subset \overline{\gamma}_iT_0$ for some $i$. 
Since $e_0$ has trivial stabilizer, this implies that $\overline{\gamma}_i=g$ and therefore $\gamma$ intersects $g\ball_X(1,r)$ as desired. 
\end{proof}

\begin{cor}\label{cor:productolibre}
Let $A,B$ be  non-trivial, finitely generated left-orderable groups. 
Then $G=A*B$ has the Hucha property with respect to $A$, $B$ and all cyclic subgroups.
\end{cor}
\begin{proof} Consider $G\curvearrowright T$ the standard action of $G$ on the Bass-Serre Tree associated to the free product. The edge stabilizers of this action are trivial, so we can apply Proposition \ref{prop: free product Hucha} which gives the corollary. (Recall that $A$ and $B$ are elliptic in this action).
\end{proof}

We now move towards the general case where the edge stabilizers are Prieto. The following is a preliminary technical lemma to that objective. 
If $X$ is generating set adapted to $(T_0,v)$, we adopt the notation $$X_v=X\cap G_v$$ Recall that this set generates $G_v$. It is illustrative to note that $\ga(G_v,X_v)$ is a connected subgraph of $\ga(G,X)$, and thus if $g_1,g_2\in G_v$ we have $d_X(g_1,g_2)\leq d_{X_v}(g_1,g_2)$.

\begin{lem}\label{lem:projection} 
Suppose that $G\curvearrowright T$ is a co-compact, minimal, general type  action of a finitely generated left-orderable group $G$ on a tree $T$. 
Let $v$ be a vertex on a compact subtree $T_0$, and $X$ a finite generating set of $G$ adapted to $(T_0,v)$ with $v$-reductions modulo a finite set $E\subseteq \text{link}_T(v)$.
Let  $P$ be a positive cone of $G$, and assume that there exists $r_0>0$ so that $P\cap G_e$ is $r_0$-connected in $\ga(G_v,X_v)$ for every $e\in E$. 

Then, given an $r$-path $\gamma=\{\gamma_i\}_{i=0}^n$ in $\ga(G,X)$ between $g_1$ and $g_2$ in $G_v$ we have one of the following:
\begin{enumerate}
\item either $\gamma$ meets $P^{-1}\ball(1,r)$,
\item or there exists $\delta=\{\delta_i\}_{i=0}^m$ a $r_0$-path in $\ga(G_v,X_v)$ supported on $P\cup\{1_G\}$ and joining $g_1$ with $g_2$.
\end{enumerate}
\end{lem}
\begin{proof} 
Take $\overline{\gamma}$ a geodesic interpolation of $\gamma$ in $\ga(G,X)$. 
Since $\gamma$ is an $r$-path, if $\overline{\gamma}$ meets $P^{-1}$ then $\gamma$ meets $P^{-1}\ball(1,r)$ and we are done. Suppose it is not the case, so $\overline{\gamma}$ is supported on $P\cup\{1_G\}$. 
Then we will find a positive $r_0$-path in $\ga(G_v,X_v)$ with the same endpoints as $\gamma$.
 
We say that an $r_0$-sequence $\delta=\{\delta_i\}_{i=0}^n$ in $\ga(G,X)$ is admissible if it is a path supported on $P\cup\{1_G\}$ with $\delta_0v=\delta_n v=v$ and satisfying that if $\dist_X(\delta_i,\delta_{i+1})>1$ then $\delta_i$ and $\delta_{i+1}$ are in $G_v$ and $\dist_{X_v}(\delta_i,\delta_{i+1})\leq r_0$.
 Informally speaking, $\delta$ may only jump inside $G_v$ and when it does, at most at distance $r_0$ with respect to $X_v$. Clearly $\overline{\gamma}$ is such an admissible path.
 We define the defect of $\delta$ as $D(\delta)=\sharp\{i:\delta_i v\neq v\}$. 
 Notice that $\delta$ is supported in $G_v$ if and only if $D(\delta)=0$. 
 Also notice that point (2) in the statement can be rephrased as: there exists an admissible path $\delta$ joining $g_1$ with $g_2$ with $D(\delta)=0$. 

We will show a procedure to reduce the defect of admissible $r_0$-paths with positive defect, while preserving the endpoints. 
Suppose that $\delta$ is  admissible  with positive defect. 
Since $D(\delta)>0$ there must exist $0\leq i<j\leq n$ such that $\delta_i v=\delta_j v=v$ and $\{\delta_k\}_{k=i+1}^{j-1}$ is a $1$-path not meeting $G_v$. 
Since $X$ has $v$-reductions modulo $E$, we deduce that $\delta_i^{-1}\delta_j\in G_e$ for some $e\in E$. Then either $\delta_i^{-1}\delta_j$ or $\delta_j^{-1}\delta_i$ belong to $G_e\cap P$ (it is clear we may assume $\delta_i\neq\delta_j$). 
Suppose that $\delta_i^{-1}\delta_j\in G_e\cap P$. 
Then, by hypothesis, we have that there exists a $r_0$-path $\alpha=\{\alpha_k\}_{k=0}^m$ in $\ga(G_v,X_v)$ supported on $P\cup\{1_G\}$ and joining $1_G$ with $\delta_i^{-1}\delta_j$.
Define $\beta$ as the concatenation of the admissible $r_0$-paths $\{\delta_k\}_{k=0}^i$ followed by $ \{\delta_i\alpha_k\}_{k=0}^m$, and finally followed by $\{\delta_k\}_{k=j}^n$. 
It is straightforward to check that $\beta$ is an admissible path with strictly smaller defect. It is also clear that the same argument works for the case when $\delta_j^{-1}\delta_i\in G_e\cap P$.

Starting with $\overline{\gamma}$ and repeating this procedure finitely many times, we obtain an admissible path with zero defect, thus finishing the proof. 
\end{proof}

We are now ready for our main Theorem.

\begin{thm}\label{thm:induction} 
Suppose that $G\curvearrowright T$ is a co-compact, minimal, general type  action of a finitely generated left-orderable group $G$ on a tree $T$. Suppose further that all vertex stabilizers are finitely generated.

Consider $v\in VT$ and $H\in\cP=\{G_e: e\in \text{link}_T(v)\}$, and assume that $G_v$ is Hucha with respect to  $H$ and all the groups of $\cP$ are Prieto. Then $G$ is Hucha with respect to the family of its cyclic subgroups, and of the subgroups acting elliptically on $T$. 
\end{thm}
\begin{proof} 
By Lemma \ref{lem:adapted} we can consider a generating set $X$ and a compact pointed tree $(T_0,v)$ so that $X$ is adapted to $(T_0,v)$. 
Consider $P$ a positive cone of $G$ with corresponding left-order $\prec$, and take $r>0$. We will find a geodesic segment $[w_1,w_2]_T$ in $T$ and a subset $S\subset P^{-1}$ so that $S$ $r$-disconnects $\cC_{w_1}$ from $\cC_{w_2}$ in $\Gamma(G,X)$. 
Since $P$ and $r$ are arbitrary, we will be able to conclude the proof by applying Lemma \ref{lem:disconecting}.   

Consider $h\in G$ so that $g\prec h$ for every $g\in\ball(1,2r)$. So we have $h^{-1}\ball(1,2r)\subset P^{-1}$, and we can also deduce that $h^{-1}(hP^{-1}h^{-1})\ball(1,2r)\subseteq P^{-1}$. 
Denote $\prec_{\ast}$ the left-order on $G$ with positive cone $P_{\ast}:=hPh^{-1}$, and re-write the last statement to get $$h^{-1}P_{\ast}^{-1}\ball(1,2r)\subseteq P^{-1}$$

By definition of an adapted generating set, there exists a finite subset $E\subseteq \text{link}_T(v)$ so that $X$ has $v$-reduction modulo $E$. 
Since $E$ is finite and the subgroups in $\cP$ are Prieto, there exists $k_0>0$ so that $P_{\ast}\cap G_e$ is $k_0$-connected in $\ga(G_v,X_v)$ for every $e\in E$. 
On the other hand, since $G_v$ is Hucha with respect to $H$, we can find $a_1,a_2\in G_v$ so that $P_{\ast}^{-1}\cap G_v$ $k_0$-disconnects $a_1H$ from $a_2H$ inside $\Gamma(G_v,X_v)$. 

Take an edge $e=[v,w]_T$ with $G_e=H$ and denote $a_1e=[v,w_1]_T$ and $a_2e=[v,w_2]_T$. 
Consider the geodesic segment $[w_1,w_2]_T$. 

\vspace{.1cm}
{\bf Claim:} $P_{\ast}^{-1}\ball(1,2r)$ $r$-disconnects $\cC_{w_1}$ from $\cC_{w_2}$. 
\vspace{.1cm}

To see this, take $h_i\in \cC_{w_i}$ for $i=1,2$ and an $r$-path ${\gamma}=\{\gamma_i\}_{i=0}^m$ in $\ga(G,X)$ joining them. 
Consider $\overline{\gamma}=\{\overline{\gamma}_{i}\}_{i=0}^n$ a geodesic interpolation of $\gamma$.
Let $\sigma\colon \{0,\ldots, m\}\to \{0,\ldots,n\}$ such that $\gamma_i=\overline{\gamma}_{\sigma(i)}$.

Since $X$ is adapted to $(T_0,v)$ we have that $\overline{\gamma}_{i+1}T_0\cap\overline{\gamma}_{i}T_0\neq\emptyset$ for $i=0,\ldots,n-1$. 
Then the definition of the $\cC_{w_i}$ implies that there exist $0<i<j<n$ so that $a_1e\in\overline{\gamma}_iT_0$ and $a_2e\in\overline{\gamma}_jT_0$. 
Note that this implies that $\overline{\gamma}_i\in a_1H$ and $\overline{\gamma}_j\in a_2H$.
If either $\overline{\gamma}_i$ or $\overline{\gamma}_j$ are in $P_{\ast}^{-1}\ball(1,r)$ then $\gamma$ meets $P_{\ast}^{-1}\ball(1,2r)$ and the claim follows. Assume the contrary, and 
let $p$ (resp. $q$) be the smallest (resp. greatest) integer such that $\sigma(p)\geq i$ (resp. $\sigma(q)\leq j$).
Consider the $r$-path $\beta=\{\beta_k\}_{p-1}^{q+1}$ with $\beta_{p-1}=\overline{\gamma}_i$, $\beta_{q+1}=\overline{\gamma}_j$ and $\beta_k=\gamma_k$ for $p\leq k\leq q$.
Note that $\beta$ is an $r$-path from $\overline{\gamma}_i\in a_1H$ to $\overline{\gamma}_j\in a_2H$, which are both included in $G_v$. We shall apply Lemma \ref{lem:projection} to the cone $P_{\ast}$ and the $r$-path $\beta$, noticing that taking $r_0 = k_0$ satisfies the hypotheses, as we deduced from the Prieto property of the subgroups in $\cP$. On the other hand, the choice of  $a_1$ and $a_2$ coming from the Hucha property of $G_v$ with respect to $H$, prevents the existence of a $k_0$-path in $\ga(G_v,X_v)$ supported on $P_{\ast}\cup\{1\}$ joining $\overline{\gamma}_i\in a_1H$ with $\overline{\gamma}_j\in a_2H$. Thus point 1 in Lemma \ref{lem:projection} holds, giving that $\beta$ meets $P_{\ast}^{-1}\ball(1,r)$. Since we are assuming the endpoints $\beta_{p-1}=\overline{\gamma}_i$ and $\beta_{q+1}=\overline{\gamma}_j$ are not in $P_{\ast}^{-1}\ball(1,r)$, we deduce that $\gamma$ meets $P_{\ast}^{-1}\ball(1,r)\subset P_{\ast}^{-1}\ball(1,2r)$. This concludes the claim. $\Diamond$


Finally, from the claim we have that $h^{-1}P_{\ast}^{-1}\ball(1,2r)$ $r$-disconnects the sets $\cC_{h^{-1}w_1}$ and $\cC_{h^{-1}w_2}$ associated to the segment $[h^{-1}w_1,h^{-1}w_2]_T$. 
Since $h^{-1}P_{\ast}^{-1}\ball(1,r)\subseteq P^{-1}$ the theorem follows from Lemma \ref{lem:disconecting}. 
\end{proof}

\begin{rem}
Note that $F_2\times \Z$ is the amalgamated free product of two groups isomorphic to $\Z^2$ along an infinite cyclic subgroup. We know that this group is not Hucha nor Prieto (Remark \ref{rem: F_2x Z}), thus the fact that at least one vertex stabilizer is Hucha is essential in Theorem \ref{thm:induction}.
\end{rem}

\begin{cor}
Let $A$ be Hucha with respect to $C\leqslant A$, where $C$ is a proper Prieto subgroup. 
For any group $B$ such that $G=A*_CB$ is left-orderable, $G$ is Hucha with respect to $\{A, B\}$ and the collection of the  cyclic subgroups.
\end{cor}
\begin{proof} The proof is analogous to that of Corollary \ref{cor:productolibre}.
\end{proof}
By a result of Bludov and Glass \cite{BludovGlass}, the amalgamation of left-orderable groups over a cyclic subgroup is again left-orderable.

\begin{cor}
Let $A$ be Hucha with respect to $C\leqslant A$ with $C$ cyclic and let $B$ be a left-orderable group.
Then $A*_CB$  is Hucha with respect to $A$ and $B$ and the collection of infinite cyclic subgroups.
\end{cor}

A one-relator group $G$ is called {\it cyclically pinched} if $G=F_1*_{\langle  a =b \rangle} F_2$ with $F_1,F_2$ finitely generated free groups, $a\in F_1-\{1\}$ and $b\in F_2-\{1\}$.
\begin{cor}
Cyclically pinched  $(\geq 3)$-generated one-relator groups are Hucha with respect to the collection of infinite cyclic subgroups.
\end{cor}

\section{A family of Hucha groups}
\label{sec: limit}

In this section we prove Theorem \ref{thm: main}.
Let $\mathfrak{H}_0$ be the family of non-abelian finitely generated free groups.
For $i>0$,  let $\mathfrak{H}_i$ be the closure under free products of the family consisting of finitely generated non-abelian subgroups, of groups of the form $G*_C A$ where $G\in \mathfrak{H}_{i-1}$, $C$ is  a cyclic centralizer subgroup of $G$ and $A$ is finitely generated abelian. 

Thus, the family $\mathfrak{H}$ of Definition \ref{def: familia H}, is equal to $\cup_{i\geq 0}\mathfrak{H}_i$.
%

\begin{prop}\label{prop:howie}
If $G\in \mathfrak{H}$ then $G$ is locally indicable (and hence left-orderable).
\end{prop}
\begin{proof}
Since every finitely generated subgroup of a free group is free, the groups in $\mathfrak{H}_0$ are locally indicable.
It is easy to see that local indicability is preserved under free products and taking subgroups.
So it remains to show that if $G\in \mathfrak{H}_n$, $C=\langle c\rangle$ is the centralizer of some $c\in G$ and $A$ is a finitely generated  abelian group, then $G*_CA$ is locally indicable.

Let $H,K$ be locally indicable groups and $r\in H*K$ an element acting loxodromically on the Bass-Serre tree of the free product.
Howie \cite{Howie} showed that a one-relator quotient $H*K/\langle \langle r \rangle \rangle$ is again locally indicable if and only if $r\in H*K$ is not a proper power (see \cite[Appendix]{ADL} for an alternative proof).
Since $G*_C A=G*A/\langle\langle ca^{-1}\rangle \rangle$ for some $a\in A-\{1\}$, and the element $ca^{-1}$ acts loxodromically (on the Bass-Serre tree of the free product) and it is not a proper power, we deduce that $G*_CA$ is locally indicable. 
\end{proof}

A group $G$ is called {\it CSA}  if all its maximal abelian
subgroups are malnormal, that means that if $A$ is a maximal abelian subgroup of $G$ and $g\in G$, then $gHg^{-1}\cap H\neq \{1\}$ implies $g\in H$.

In \cite{MiRe} define the class CSA$^*$ as the CSA groups with no elements of order 2. It is proved that CSA$^*$ is closed under free products \cite[Theorem 4]{MiRe}, and the following construction: if $L$ is CSA$^*$, $C$ is a centralizer of an element of $L$ and $A$ is torsion-free finitely generated abelian, then $L*_CA$ is CSA$^*$ \cite[Theorem 5]{MiRe}.
Since CSA clearly passes to subgroups and free groups are CSA we have.
\begin{lem}\label{lem: CSA}
If $G\in \mathfrak{H}$ then $G$ is CSA.
\end{lem}

Recall that  a simplicial $G$-tree is $k$-acylindrical if the fixed point set of any $g\in G$ has diameter at most $k$ (i.e. any set of diameter $>k$ has trivial stabilizer).
\begin{lem}\label{lem: acylindrical}
If $T$ is the Bass-Serre tree of $G=A*_CB$ with $C$ malnormal in $A$ then the action is 2-acylindrical. 
\end{lem}
\begin{proof}
Indeed, the stabilizer of an edge $e \in ET$ is a conjugate of $C$ (we can assume by $G$-equivariance that it is $C$). 
Let $v$ be the vertex adjacent to $e$ with stabilizer $L$.
All edges adjacent to $v$ have stabilizer $lCl^{-1}$ with $l\in L\setminus C$.
By malnormality of $C$ in $L$, $C\cap lCl^{-1}$ is trivial.
Thus $C$ does not fix any other edge adjacent to $v$ different from $e$.
Now any subset of diameter $\geq 3$, up to $G$-equivariance, contains a path of length $\geq 2$ that has $v$ as an internal vertex and hence it has trivial stabilizer.
\end{proof}

We need the following fact.

\begin{lem}\cite[Lemma 2]{Cohen}\label{lem:fg-vertex groups}
Let $G$ be a finitely generated group acting co-finitely on a tree $T$. 
If the stabilizers of edges adjacent to $v\in VT$ are finitely generated, then the stabilizer of $v$ is finitely generated.
\end{lem}


We can prove now our main theorem.

\begin{proof}[Proof of Theorem \ref{thm: main}]
We will prove by induction that groups in $\cup_{i=0}^n\mathfrak{H}_i$ are Hucha with respect to the family of cyclic centralizers.  The base case of induction ($n=0$) is a consequence of Proposition \ref{prop:F2hucha}.
Thus assume that the induction hypothesis hold for $n>0$. Have to prove the case $n+1$.

Let $L\in \mathfrak{H}_{n}$, $C\leqslant L$ a subgroup of a cyclic centralizer and $A$ a finitely generated abelian group.
Without loss of generality we assume that $C$ is a proper subgroup of $A$ (since if not $L*_CA=L$).
Let $T$ be the associated Bass-Serre tree to $L*_CA$.
By Lemma \ref{lem: CSA}, $C$ is malnormal in $L$, which implies (Lemma \ref{lem: acylindrical}) that the action on the $T$ is $2$-acylindrical.
In particular, if $H$ is a subgroup $L*_CA$ with elements acting loxodromically on $T$ and such that the minimal $H$-tree is homeomorphic to $\mathbb{R}$, then it must be virtually cyclic (see for example \cite[Theorem 1.1.]{Osin}), and since $L*_C A$ is torsion-free (Proposition \ref{prop:howie}), $H$ is must be cyclic.

Let $G\leqslant L*_C A$ be a non-abelian finitely generated subgroup. 
We have to show that $G$ is Hucha with respect to its cyclic centralizers.

If $G$ is, up to conjugation, a subgroup of $L$. 
Then, by induction hypothesis $G$ is Hucha with respect to its cyclic centralizers. 
So we restrict to the case where $G$ is not contained, up to conjugation, in $L$.
Note that, since $G$ is non-abelian, it is not contained, up conjugation, in $A$.

Thus the action of $G$ on the Bass-Serre tree of $L*_CA$ has no-global fixed point.
By \cite[Proposition I.4.11]{DicksDunwoody} $G$ contains elements acting loxodromically on $T$.
Consider $T'\subseteq T$ the minimal $G$-invariant subtree of $T$.
This is the subtree consisting on the union of the axis of the elements of $G$ acting loxodromically. 
Since $G$ is finitely generated, $G\backslash T'$ is finite
(see \cite[Proposition I.4.13]{DicksDunwoody}).
By the previous discussion, if $T'$ has two ends,  $G$ must be cyclic, contradicting that $G$ is non-abelian.
In fact, since the action of $G$ of $T'$ is acylindicral, any element acting loxodromically on $T$ has cyclic stabilizer.
Finally, since $L*_C A$ does not fix any end of $T$, $G$ does not fix any end of $T'$.
Thus, the action of $G$ on $T'$ is of general type and co-compact, with edge stabilizers that are trivial or cyclic, and vertex stabilizers that are isomorphic to subgroups of $A$ or subgroups of $L$.
Moreover, by Lemma \ref{lem:fg-vertex groups} the vertex groups are finitely generated.


If $G$ acts on $T'$ with trivial vertex stabilizers, then $G$ is free, and thus $G\in \mathfrak{H}_0$ and there is nothing to prove.

Suppose that $G$-stabilizer of all vertices of $T'$ are non Hucha. 
Thus, all $G$-stabilizers of all vertices of $T'$ must be abelian, since by induction, all non Hucha finitely generated subgroups of $L$ are the abelian ones.
If some edge has trivial stabilizer, by Proposition \ref{prop: free product Hucha}, $G$ is Hucha with respect to the subgroups acting loxodromically on $T'$ and the subgroups of the vertex stabilizers. In particular it is Hucha with respect to the cyclic centralizers.
Thus we can assume that all edge stabilizer are non-trivial and cyclic.
Let $e$ be an edge of $T'$ adjacent to $u$, and suppose that the stabilizer of $u$ is a conjugate of a subgroup of $L$.
Since $G_u$ is abelian and $G_e\leqslant G_u$ is a centralizer of some element (restricted to $G_u$) we get that $G_u=G_e$.
We can $G$-equivariantly collapse the edge $e$ and we will still have an action on a tree with the previous properties, and we could repeat this argument and continue $G$-equivariantly collapsing edges, obtaining that $G$ is abelian, a contradiction.

Thus the only case remaining is when there is some $v\in T'$ with $G_v$ an Hucha group.
Thus $G_v$ must be  (up to conjugation) a finitely generated subgroup of $L$ and thus, by induction, it is Hucha with respect to all its cyclic centralizers. 
Since the stabilizers of edges adjacent to $v$ are either trivial or Prieto, we can use Theorem \ref{thm:induction} to conclude that $G$ is Hucha with respect to all its cyclic centralizers.
\end{proof}

\paragraph{Application to limit groups.}
There are several equivalent definitions in the literature for {\it limit groups.} 
See for example \cite{ChampGuir} for a nice survey. The quickest definition is that $G$ is a limit group if it is a finitely generated fully residually free group, meaning that for any finite set $S\subseteq G$ there is a homomorphism $\phi \colon G\to F$ where $F$ is a free group and $\phi$ is injective restricted to $S$. 
However, to show that limit groups belong to $\mathfrak{H}$ we need a different characterization.

Let $H$ be a group and $Z$ the centralizer of some non-trivial element of $H$. 
A {\it free extension of $H$ by the centralizer $Z$} is a group of the form $H*_Z(Z\times A)$ for some finitely generated free abelian group $A$. 

Let $\textrm{ICE}_0$ denote the class of finitely generated free groups. For $i>0$, let $\textrm{ICE}_i$ denote the groups that are free extensions by cyclic centralizers of groups in $\textrm{ICE}_{i-1}$. Finally let $\textrm{ICE}$ denote the union of $\textrm{ICE}_i$ $i\in \mathbb{Z}_\geq 0$.

An important and very useful result for us is the following.
\begin{thm}\cite[Thm 4.]{KM98}.\label{thm:KM} The following are equivalent
\begin{enumerate}
\item $G$ is a limit group 
\item $G$ is a finitely generated subgroup of a group in  $\textrm{ICE}$
\end{enumerate}
\end{thm}

Clearly, in view of Theorem \ref{thm:KM} we have that non-abelian limit groups are in $\mathfrak{H}$.

\paragraph{Application to free $\Q$-groups.}
Recall that $H$ is a $\Q$-group if for every $h\in H$ and all $n\in \N$ there is a unique $g\in H$ such that $g^n=h$. 
This allows to define an action of $\Q$ on $H$, and we denote the image of the action of $\alpha\in \Q$ on $h\in H$ by $h^\alpha$. 
If $\alpha=p/q$ then $h^\alpha$ is $g^p$ where $g$ is the unique element such that $g^q=h$.
It is easy to check that indeed this in an action.

A $\Q$-group $G^\Q$ together with an homomorphism $\phi \colon G \to G^\Q$ is called the {\it tensor $\Q$-completion of the group $G$} if it satisfy the following universal property: 
for any $\Q$-group $H$ and a homomorphism $f \colon G\to H$ 
there exists a unique $\Q$-homomorphism $\psi \colon G^\Q \to H$ (a homomorphism that commutes with the $\Q$-action) such that $f=\psi \circ \phi$.

$\Q$-groups were introduce by G. Baumslag in \cite{GBaumslag} where he showed that $\Q$-tensor completions exists and are unique. 
A free $\mathbb{Q}$-groups  is the tensor $\Q$-completion of a free group $F$.
Crucially, Baumslag also showed that $F^\Q$ can be obtained from a free group $F$ by interatively adding roots to $F$ i.e.
there is chain
$$F=F_0<F_1<F_2< \dots $$
such that $F_{n+1}$ is obtained from $F_{n}$ by adding a some $q$-th root of an element generating its own centralizer i.e. if $c\in F_n$ and $\langle c \rangle =\{g\in F_n \colon gc=cg\}$ then $F_{n+1}=F_n*_{c=t^q}\langle t \mid \;\rangle$.
There are some subtleties about how to construct this ascending chain (see \cite[Section 8]{MiRe}) however, clearly any finitely generated subgroup of a $\Q$-free group is a subgroup of a finite iterative addition of roots to a free group $F$ and hence it lies in $\mathfrak{H}$ \footnote{Since finitely generated subgroups of free $\Q$-groups are in $\mathfrak{H}$, Proposition \ref{prop:howie} then implies that  free $\Q$-groups are left-orderable. In fact, G Baumslag ask  whether these groups are residually torsion free nilpotent a fact recently proved by A. Jaikin \cite{Jaikin}. In particular, free $\Q$-groups are bi-orderable, that is they admit left-orders which are also right-invariant. }.

\begin{ex}
An example of a group in $\mathfrak{H}$ is $G=\langle a,b,c\mid a^2b^2c^2\rangle$, the fundamental group of the connected sum of 3-projective planes. 
It follows from a theorem of Lyndon \cite{Lyndon} that  the equation $a^2b^2c^2=1$ in a free group implies that $a,b,c$ commute. 
Hence $G$ is not a limit group. 
On the other hand, $G$ is obtained from the free group $\langle a,b\mid \;\rangle$, by adding a square root to $a^2b^2$ and hence it is a finitely generated subgroup of a free $\Q$-group.
\end{ex}

\section{Regular sets are coarsely connected}
\label{sec: regular}

Let $X$ be a set. 
Recall that $X^*$ denotes the {\it free monoid generated by $X$} and consists on the set of finite words on $X$ together with concatenation. If $X\subseteq G$ generates the group $G$, then there is a natural monoid epimorphism $\mathrm{ev}\colon X^*\to G$, called {\it evaluation map}, that is induced by viewing each element of $x\in X$ as element of $G$.

\begin{defi}\label{def:regular language}
A {\it finite state automaton} is a 5-tuple $(\cS,\cA,s_0,X,\tau)$, where $\cS$ is a set whose elements are  called {\it states}, $\cA$ is a subset of $\cS$ of whose states are called {\it accepting states}, a distinguished element  $s_0\in \cS$ called {\it initial state}, a finite set $X$ called the {\it input alphabet} and a function $\tau\colon \cS\times X\to \cS$ called {\it the transition function}. We extend $\tau$ to a function $\tau\colon \cS\times X^*\to \cS$ recursively, by setting $\tau(s,wx)=\tau(\tau(s,w),x)$ where $w\in X^*$, $x\in X$ and $s\in \cS$. 

A {\it language $\cL$ over $X$} is a subset of $X^*$, and $\cL$ is {\it regular} if there is a finite state automaton $(\cS,\cA,s_0,X,\tau)$
such that $$\cL=\{w\in X^* \mid \tau(s_0,w)\in \cA\}.$$
\end{defi}

\begin{prop}\label{prop:r-connected}
Let $G$ be a group generated by a finite set $X$. 
If $\mathcal{L}\subseteq X^*$ is a regular language, then $\mathrm{ev}(\mathcal{L})$ is an $r$-connected subset of $\ga(G,X)$ for some $r$.
\end{prop}
\begin{proof}
Assume that $\cL$ is a regular language over $X$ and $(\cS,\cA,s_0,X,\tau)$ is a finite state automaton accepting $\cL$. 
We can view the automaton as a directed graph with vertices $\cS$ and edges $\cS\times X$ where $(s,x)$ is an edge from $s$ to $\tau(s,x)$. 
We call this graph, the {\it directed automaton graph}.
Let $r=|\cS|$.
Let $w\equiv x_1x_2\dots x_n\in \cL$. 
We will show that we can go from $1$ to $\mathrm{ev}(w)$ with an $(2r+1)$-path $\{v_i\}_{i=0}^n$ where $v_i\in \mathrm{ev}(\cL)$.
Indeed, let $w_i$ be the prefix of length $i$ of $w$, i.e. $w_i=x_1x_2\dots x_i$. 
It is enough to find $v_i\in \mathrm{ev}(\cL)$ such that $\dist(\mathrm{ev}(w_i), v_i)\leq r$.
Note that $w_i$ is a prefix of some word in $\cL$, so we take a word of minimal length $w'_i$ with the property that $w_i w_i'\in \cL$.
Now, $w_i'$ gives a path in directed automaton graph from $\tau(s_0, w_i)$ to some vertex in $\cA$.
By minimality, this path can not repeat a vertex and thus $\ell(w_i')\leq r$.
Note that $w_i'$ will be the label of some path in the Caley graph from $\mathrm{ev}(w_i)$ to $\mathrm{ev}(w_iw_i')$ of length at most $r$. 
Thus we can take $v_i$ to be  $\mathrm{ev}(w_iw_i')$.
\end{proof}

\begin{defi}
A positive cone $P$ of a finitely generated group $G$ is {\it regular} if there is a finite generating set $X$ of $G$ and a regular language $\cL$ over $X$ such that $P=\mathrm{ev}(\cL)$.
\end{defi}

\begin{rem}
Finitely generated (as subsemigroups) positive cones are regular. 
\end{rem}

\begin{cor}\label{cor: hucha not regular}
If $G$ is Hucha, then $G$ has no regular positive cone. 
\end{cor}
\begin{proof}
Let $P$ be a positive cone for $G$ and $X$ a finite generating set of $G$.
Suppose that there is a regular language $\mathcal{L}\subseteq X^*$ such that the evaluation of $\mathcal{L}$ in $G$ is $P$. 
By Proposition \ref{prop:r-connected}, there is some $r>0$ such that $P$ is an $r$-connected subset of $\ga(G,X)$.  
This contradicts Proposition \ref{prop: Hucha are not coarsely connected}.
\end{proof}

\begin{rem}\label{rem: combing}
Let $G$ be a group and $X$ a generating set.
If $\cL$ is a regular language over $X$ such that $\mathrm{ev}(\cL)$ is a positive cone $P$, 
then the set of paths in $\Gamma(G,X)$ starting at $1_G$ and with label $w\in \cL$ is a combing of $P$ and by Proposition \ref{prop:r-connected} this combing is supported in the $r$-neighbourhood of $P$ for some $r$.
\end{rem}

Suppose that $X$ is a finite generating set of a group $G$. 
A language $\cL\subseteq X^*$ is {\it quasi-geodesic} if there are $\lambda\geq 1$ and $c\geq 0$ such that each word $w\in \cL$ labels a $(\lambda,c)$-quasi-geodesic path.

In \cite[Question 8.7]{calegari} Calegari sketched an argument for showing that on the fundamental group of an hyperbolic 3-manifold, no positive cone language is both regular and  quasigeodesic. Recently H. L. Su \cite{Su} showed, using Calegari's ideas, that in fact no positive in the (much larger) class of acylindrically hyperbolic group can be described by a regular and quasigeodesic language. With our tools we can easily recover Su's theorem in restriction to the class of $\delta$-hyperbolic groups. 

\begin{cor}
Non-elementary hyperbolic groups do not admit regular quasi-geodesic positive cones.
\end{cor}
\begin{proof}
By Remark \ref{rem: combing}, if $P$ is a regular quasi-geodesic positive cone, then it admits a quasi-geodesic combing supported on an $r$-neighbourhood of $P$ for some $r$. This contradicts Theorem \ref{thm combing}
\end{proof}


\begin{small}


\vspace{1.3cm}

\end{small}


\textit{Juan Alonso}

Fac. Ciencias, Universidad de la Republica Uruguay

juan@cmat.edu.uy

\bigskip

\textit{Yago Antol\'{i}n}

Fac. Matem\'{a}ticas, Universidad Complutense de Madrid \& ICMAT

yago.anpi@gmail.com

\bigskip

\textit{Joaquin Brum}

Fac. Ingenieria, Universidad de la Republica Uruguay

joaquinbrum@gmail.com

\bigskip

\textit{Crist\'obal Rivas}

Dpto. de Matem\'aticas y C.C., Universidad de Santiago de Chile

cristobal.rivas@usach.cl

\end{document}